\documentclass[]{interact}

\usepackage{caption}
\usepackage{subcaption}
\usepackage{algorithm}
\usepackage{algpseudocode}

\newcommand{\zjj}[1]{{\color{purple}#1}}
\newcommand{\eb}[1]{{\color{teal}#1}}
\DeclareMathOperator{\tr}{tr}

\restylefloat{table}
\usepackage{color}
\usepackage[table]{xcolor}
\usepackage{colortbl}
\usepackage{epsfig} 
\usepackage{times} 
\usepackage{amsmath} 
\usepackage{amssymb}  
\usepackage{enumerate}
\usepackage{cite}

\usepackage{color}
\usepackage{etoolbox}
\usepackage{graphicx}

\usepackage{float}
\DeclareGraphicsExtensions{.png}
\usepackage{thmtools}
\usepackage{epstopdf}%
\usepackage[final]{changes}

\usepackage[numbers,sort&compress]{natbib}
\bibpunct[, ]{[}{]}{,}{n}{,}{,}
\renewcommand\bibfont{\fontsize{10}{12}\selectfont}

\theoremstyle{plain}
\newtheorem{theorem}{Theorem}[section]
\newtheorem{lemma}[theorem]{Lemma}

\newtheorem{proposition}[theorem]{Proposition}
\newtheorem{assumption}{Assumption}

\theoremstyle{definition}

\theoremstyle{remark}
\newtheorem{remark}{Remark}

\begin{document}

\articletype{}

\title{Soft quasi-Newton:\\ Guaranteed positive definiteness by relaxing the secant constraint}

\author{
\name{Erik Berglund, Jiaojiao Zhang and Mikael Johansson \thanks{CONTACT Erik Berglund. Email: erbergl@kth.se}}
\affil{School of Electrical Engineering and Computer Science, KTH Royal Institute of Technology, Sweden}
}

\maketitle

\begin{abstract}
%
%

We propose a novel algorithm, termed soft quasi-Newton (soft QN), for optimization in the presence of bounded noise. Traditional quasi-Newton algorithms are vulnerable to such perturbations. To develop a more robust quasi-Newton method, we replace the secant condition in the matrix optimization problem for the Hessian update with a penalty term in its objective and derive a closed-form update formula. A key feature of our approach is its ability to maintain positive definiteness of the Hessian inverse approximation. Furthermore, we establish the following properties of soft QN: it recovers the BFGS method under specific limits, it treats positive and negative curvature equally, and it is scale invariant. Collectively, these features enhance the efficacy of soft QN in noisy environments. For strongly convex objective functions and Hessian approximations obtained using soft QN, we develop an algorithm that exhibits linear convergence toward a neighborhood of the optimal solution, even if gradient and function evaluations are subject to bounded perturbations. Through numerical experiments, we demonstrate superior performance of soft QN compared to state-of-the-art methods in various scenarios.
\end{abstract}

\begin{keywords}
quasi-Newton methods; general bounded noise; secant condition; penalty
\end{keywords}

\section{Introduction}\label{sec:Introduction}

Although the importance of uncertainty and stochasticity has been recognized since the early days of numerical optimization (see, \emph{e.g.},~\cite{RM_SGD,Dan:55}), it is fair to say that the majority of solution algorithms have been developed under the assumption that the objective and constraint functions (and often also their gradients and Jacobians) can be evaluated accurately. Even stochastic programs~\cite{BiL:97}, that explicitly model the uncertainty in the underlying decision problem, are often approached via their deterministic equivalents. Techniques for analyzing the effect of noise on iterative algorithms have a rich history (see, \emph{e.g.}~\cite{Polyak}) but 
it is only over the last few decades that the research devoted to algorithms that operate reliably using noisy function and gradient evaluations has virtually exploded. This increased interest has been fueled by an increasing number of applications in machine learning and simulation optimization. 

Machine learning, a subset of artificial intelligence, automates pattern discovery within data sets and leverages these patterns to make inferences based on input data \cite{jordan2015machine}. Optimization plays a foundational role in various machine learning approaches as they aim to minimize the discrepancy between model predictions and observed data, necessitating the solution of optimization problems \cite{sra2012optimization}. A particularly challenging setting in machine learning arises when dealing with large datasets that cannot be processed simultaneously. In such cases, practitioners employ data sampling techniques to construct approximations of the objective function to be minimized, rather than evaluating it exhaustively \cite{zhang2004solving, bottou2010large}. This situation leads to the formulation of stochastic optimization problems.
In addition to settings where data is sampled, stochastic optimization problems can arise from numerical simulations that contain stochastic components. Simulation optimization, i.e. optimization of functions that can only be evaluated through a stochastic simulation, is an important area with applications in diverse domains such as operations research, manufacturing, medicine, and engineering~\cite{Amaran2016}. In these simulations, the statistical properties of the noise are not as clear, and it is natural to consider the errors as bounded noise
with otherwise unknown properties. 

First-order optimization methods are popular for solving stochastic problems. In particular, the stochastic gradient descent method~\cite{RM_SGD} and variations thereof have become the workhorse of modern large-scale machine learning \cite{OptimizationML}. The scalability and ease of implementation of these methods contribute to their popularity. In addition, they can achieve the optimal oracle complexity up to a multiplicative constant in many stochastic settings \cite{6142067}. However, first-order methods struggle on ill-conditioned and highly nonlinear problems. For example, the iteration complexity of stochastic gradient descent on  Lipschitz smooth, strongly convex problems grows proportionally to the condition number~\cite{OptimizationML}. In the non-convex setting, empirical evidence suggests that second-order methods have advantages when it comes to the robustness of hyper-parameter tuning and the ability to escape saddle points \cite{Second_order_non_convex_empirical} \cite{10.5555/2969033.2969154} \cite{Paternain2019NewtonNonconvex}. 

While the use of Hessian information can significantly improve the iteration complexity of optimization algorithms, it can often be expensive to compute. To gain some of the benefits of Newton's method in settings where only gradient information is available, quasi-Newton (QN) methods build a model of the Hessian using only first-order information. In the deterministic setting, quasi-Newton methods have proven to be much more effective than gradient descent for small- and medium-scale problems, both in theory and practice \cite{NoceWrig06}. For large-scale problems, limited memory quasi-Newton methods, and L-BFGS \cite{Liu1989} in particular, have been prominent. However, traditional quasi-Newton methods are sensitive to the presence of random noise in the gradients. Thus, some adaptions must be made to retain their advantages in the stochastic setting. Currently, this is a very active topic of research.

A significant effort has been devoted to adapting quasi-Newton methods to problems where the stochastic noise comes from the sampling of gradients. These techniques combine stochastic gradients with carefully constructed Hessian estimations. For example, the paper~\cite{bordes2009sgd} constructs diagonal or low-rank Hessian approximations originating from the secant condition. The study~\cite{byrd2016stochastic} decouples the stochastic gradient and curvature estimate calculation by replacing the gradient difference that is typically used in quasi-Newton methods with a Hessian-vector product. A different approach presented in~\cite{mokhtari2014res} introduces a regularized BFGS method that incorporates regularization techniques to ensure positive definiteness and boundedness of Hessian estimations for strongly convex loss functions. For non-convex problems, the paper   \cite{wang2017stochastic} employs damping techniques to preserve the positive definiteness and boundedness of Hessian estimations. Another set of studies \cite{mokhtari2018iqn,lahoti2023sharpened} investigate incremental quasi-Newton approaches, wherein a solitary Hessian approximation associated with a chosen objective loss from the finite sum is updated, while all other Hessian approximations remain unchanged. 

Another extension of stochastic gradient descent (SGD) for constructing descent directions is based on combining variance reduction techniques with quasi-Newton algorithms to mitigate the noise impact introduced by stochastic sampling \cite{moritz2016linearly,yasuda2022addhessian,chen2023modified}.

In the setting of stochastic optimization with general bounded noise, quasi-Newton methods have received considerably less attention. A study by Xie et al. analyzed the behavior of BFGS in the presence of errors in the function and gradient evaluations \cite{Xie2020analysis}, and showed that with a lengthening technique to modify the line-search procedure, BFGS can be made to converge linearly to a neighborhood of the minimum of a strongly convex function. This work was expanded upon in \cite{shi2022noise}, where a modification of the line-search method was proposed that did not require knowledge of the strong convexity constant. Similar ideas were explored earlier in the derivative-free setting, where the length of finite-difference interval was used to control the influence of noise in quasi-Newton approximations \cite{Berahas2019DerivativeFree}.
Apart from adjusting the finite-difference approximation to deal with bounded noise, another idea is to relax the secant condition because, in the presence of noise, enforcing the secant condition for quasi-Newton updates can be problematic. A recent work on stochastic quasi-Newton methods takes this approach \cite{Irwin2023}. By replacing the secant condition in the problem used to derive BFGS with a penalty term, the authors derive a method which they call secant-penalized BFGS, or SP-BFGS for short. They show that the resulting update formula can guarantee positive definiteness, provided that the penalty parameter is small enough. 

In this paper, we introduce a novel quasi-Newton algorithm for minimizing a loss function in the presence of general bounded noise. Similarly to \cite{Irwin2023}, we relax the secant condition by replacing it with a penalty term, but we consider a different matrix optimization problem and derive a new Hessian approximation with several attractive features. We refer to the proposed approach as soft QN, as it ``softens" the constraint of having to satisfy the secant condition. Unlike SP-BFGS, which requires its penalty parameter to be sufficiently small to ensure positive definiteness of the Hessian approximation, we prove that soft QN generates a positive definite matrix for all positive penalty parameters, even when the loss functions are non-convex. In addition, we establish several other attractive properties of soft QN, including that it recovers BFGS under specific limits, handles positive and negative curvature equally, and exhibits scale invariance. These characteristics enhance the efficacy of the proposed soft QN in noisy settings. For strongly convex and Lipschitz smooth loss functions, the resulting quasi-Newton algorithm generates iterates that are guaranteed to decrease the objective function at a linear rate to a neighborhood of the optimal solution.  Numerical experiments show that soft QN is competitive with SP-BFGS and superior to other benchmark algorithms.

\section{Preliminaries: quasi-Newton methods}\label{sec:QN_methods}

In this section, we provide a brief review of the key ideas that underpin traditional quasi-Newton methods~\cite{NoceWrig06}. For clarity, we focus on unconstrained minimization problems
\begin{align}\label{eq:objective}
	\begin{array}[c]{ll}
		\underset{x\in \mathbb{R}^n}{\mbox{minimize}} & f(x),
	\end{array}
\end{align}
where the objective function $f: \mathbb{R}^n\mapsto \mathbb{R}$ is continuous and twice differentiable. 

In essence, quasi-Newton methods are iterative optimization algorithms that use approximations of the Hessian of the objective function to speed up their convergence. These approximations are constructed based on the previous iterates $(x_k)$ generated by the method and the gradient of the objective at these points. More specifically, given an initial point $x_0$, these methods generate a sequence of iterates
\begin{align}
	x_{k+1} &= x_k + \eta_k p_k \label{eqn: variable_update},
\end{align}
where $\eta_k>0$ is a step-size and $p_k \in \mathbb{R}^n$ is a search direction chosen to minimize a quadratic model $q_k$ of $f$ around $x_k$. The model uses $f(x_k)$, $\nabla f(x_k)$ and an approximation $B_k$ of $\nabla^2 f(x_k)$: 
\begin{align*}
	q_k(p) &= f(x_k) + \nabla f(x_k)^{\top}p + \frac{1}{2}p^{\top}B_k p.
\end{align*}
When $B_k$ is invertible, $q_k$ has the unique critical point
\begin{equation}\label{eq:pk}
	p_k = - B_k^{-1}\nabla f(x_k).
\end{equation}
If $B_k$ is positive definite the model is minimized at this critical point and $p_k$ is a descent direction for $f(x_k)$, \emph{i.e.} for a sufficiently small step-size $\eta_k$, we have $f(x_k + \eta_k p_k) < f(x_k)$. Thus, positive definiteness is a desirable property for $B_k.$ 

Different Hessian approximations $B_k$ lead to different quasi-Newton methods. Most methods use Hessian approximations that satisfy the secant condition
\begin{equation}\label{eqn:secant_condition}
	y_k= B_{k+1} s_k,
\end{equation}
where $s_k:= x_{k+1} - x_k$ and $y_k := \nabla f(x_{k+1}) - \nabla f(x_k)$.
The condition is often motivated by the fact that it is satisfied by the actual Hessian of the loss function. Specifically, since
\begin{equation*}
	y_k = \int_0^1 \nabla^2 f(x_k + t s_k)s_k dt,
\end{equation*}
the mean value theorem of integrals says that there exists a $\theta \in [0,1]$ such that 
\begin{equation*}
	y_k = \nabla^2 f(x_k + \theta s_k) s_k.
\end{equation*}
Requiring Equation~\eqref{eqn:secant_condition} to be satisfied adds information about the true objective function's curvature to the model, but still leaves several degrees of freedom for the choice of $B_{k+1}$. A key principle of quasi-Newton methods is that they update $B_{k+1}$ from the previous Hessian approximation $B_k$ by minimizing the size of the necessary modification~\cite{NoceWrig06}. This allows the algorithm to accumulate approximate curvature information from previous iterations and aids convergence when close to a local minimum, where the loss function can be well-approximated by a quadratic one. Different metrics for measuring closeness between $B_{k+1}$ and $B_k$ lead to different quasi-Newton methods.

The most well-known quasi-Newton method, BFGS, has strong practical performance and its limited-memory variant, L-BFGS, is used in various large-scale optimization solvers \cite{NoceWrig06}. To derive the BFGS update, one considers the inverse of the Hessian approximation, $H_k$, and chooses $H_{k+1}$ to satisfy the inverse secant condition
\begin{equation}\label{eqn:inverse_secant_condition}
	H_{k+1}y_k = s_k.
\end{equation}
To ensure that $H_{k+1}$ is close to $H_k$, one considers the matrix optimization problem
\begin{equation}\label{eqn:BFGS_problem}
	\begin{aligned}
		& \underset{H}{\text{minimize}}
		& & \|W_k^{(1/2)}(H - H_k)W_k^{(1/2)}\|_F^{2} \\
		& \text{subject to} & & H y_k = s_k,\; H = H^T.
	\end{aligned}
\end{equation}
Here, $\| \cdot\|_F$ denotes the Frobenius norm, $W_k$ is any positive definite matrix that satisfies $W_k s_k = y_k$, and $W_k^{(1/2)}$ is its principal square root. Setting $H_{k+1}=H^{\star}$, where $H^{\star}$ is the solution to \eqref{eqn:BFGS_problem}, leads to the update formula
\begin{align}\label{eqn:BFGS_H_update}
	H_{k+1} &= \left(I - \dfrac{s_k y_k^T}{y_k^T s_k} \right) H_k \left(I - \dfrac{y_k s_k^T}{y_k^Ts_k} \right) + \frac{s_k s_k^T}{y_k^T s_k}.
	\intertext{For a derivation of this, see e.g. \cite{variationalQN}. By the Sherman-Morrison-Woodbury formula, the corresponding update for $B_k$ is}
	B_{k+1} &= B_k - \dfrac{B_k s_k s_k^T B_k}{s_k^T B_k s_k} + \dfrac{y_k y_k^T}{y_k^T s_k}.\label{eqn:BFGS_B_update}
\end{align}
An important property of the BFGS update is that if $H_k$ is positive definite, then $H_{k+1}$ will be positive definite if and only if $s_k^{\top}y_k>0$ \cite[p. 140]{NoceWrig06}. In fact, a positive definite weight matrix $W_k$ satisfying $W_ks_k=y_k$ exists if and only if $s_k^{\top}y_k>0$, so the matrix optimization problem (\ref{eqn:BFGS_problem}) is only defined when this curvature condition holds.

The BFGS update also solves the matrix optimization problem
\begin{equation}\label{eqn:BFGS_problem_2}
	\begin{aligned}
		& \underset{B \succ 0}{\text{minimize}}
		& & \tr (B B_k^{-1}) - \log(\det(B B_k^{-1})) \\
		& \text{subject to} & & B s_k = y_k.
	\end{aligned}
\end{equation}
In this problem, the weighted Frobenius norm has been replaced by the log-det divergence to measure the distance between $B$ and $B_k$.
Since the constraints on $B$ require that $B \succ 0$ and that $B s_k = y_k$, it only admits a solution if $s_k^T y_k > 0$. 

In the deterministic setting, BFGS is typically combined with a line-search method to determine the step-size $\eta_k$ in Equation~\eqref{eqn: variable_update}. If the line-search technique enforces the Wolfe conditions \cite[Equations (3.6a) and (3.6b)]{NoceWrig06}, the algorithm ensures not only descent in each iteration but also that $s_k^T y_k > 0$ holds. In the stochastic setting, where $y_k$ and $s_k$ are perturbed by noise, it may be impossible to satisfy this condition. This contributes to making regular quasi-Newton methods very sensitive to noise, to the extent that even a slight perturbation in the gradients might cause them to break down.
%

As outlined in Section~\ref{sec:Introduction}, multiple attempts have been made at adapting quasi-Newton methods for stochastic problems. In the next section, we introduce a method that modifies the matrix optimization problem that defines the BFGS update to be more robust to noise, yet results in a simple update formula for the  Hessian inverse approximation.

\section{The soft QN method}
\label{sec:softQN}
In this section, we introduce soft QN, a quasi-Newton method derived by solving a matrix optimization problem with a ``softened'' secant condition.
The resulting algorithm is similar to the classical BFGS method but it is more robust to noise and guaranteed to always yield a positive definite  Hessian inverse approximation.
Below, we pose the matrix optimization problem, derive the corresponding update formula, and discuss key properties of the method.

\subsection{A relaxed matrix optimization problem and its solution}

As discussed above, if $s_k^T y_k \leq 0$ there is no matrix that is both positive definite and satisfies the secant equation~\eqref{eqn:secant_condition}. Therefore, the matrix optimization problems \eqref{eqn:BFGS_problem} and \eqref{eqn:BFGS_problem_2} that are traditionally used for deriving quasi-Newton updates are only justified when $s_k^Ty_k > 0$. However, suppose that one replaces the secant constraint in those problems with a penalty term that punishes deviations from the secant condition. It may then be possible to obtain a positive definite $H_{k+1}$ that is close to satisfying the secant equation. 
%
To this end, we consider the following 
relaxation of the matrix optimization problem (\ref{eqn:BFGS_problem_2}) 
\begin{equation}\label{eqn:soft_QN_problem}
	\begin{aligned}
		& \underset{B}{\text{minimize}}
		& \Upsilon(B):= & \ \tr (BH_k) - \log(\det(B H_k)) + \alpha_k 
		\upsilon(s_k-B^{-1}y_k)
		\\
		& \text{subject to} & &  B \succ 0.
	\end{aligned}
\end{equation}
The penalty function $\upsilon: \mathbb{R}^n\mapsto\mathbb{R}_{\ge 0}$ replaces the hard secant condition and  
is assumed to be convex with $\upsilon(0)=0$. The parameter  $\alpha_k\geq 0$ 
is used to determine how large a penalty is imposed for violating the secant equation. There are various ways that one could choose the form of the penalty term. As we will show later in this section, letting
\begin{equation}\label{eqn:penalty_term}
	\upsilon(s_k-B^{-1}y_k) = \Vert s_k-B^{-1}y_k\Vert_B^2,
\end{equation}
where $\Vert z\Vert_B^2=z^T Bz$, leads to an explicit update formula that is invariant under linear changes of variables. Furthermore, since this particular penalty function can be written as both $(s_k - B^{-1}y_k)^TB(s_k - B^{-1}y_k)$ and $(B s_k - y_k)^TB^{-1}(B s_k - y_k)$, it can be seen as a penalty on both the residual of the secant condition \eqref{eqn:secant_condition} and of the inverse secant condition \eqref{eqn:inverse_secant_condition}. 
The main result of this subsection is given by the following theorem.
\begin{theorem}
	For  every $\alpha_k>0$ and every $H_k \succ 0$, there exists a unique positive definite solution $B^{\star}$ to~\eqref{eqn:soft_QN_problem} with the function $\upsilon$ defined in~\eqref{eqn:penalty_term}. Letting $H_{k+1}=(B^{\star})^{-1}$ leads to the recursive update
	%
	\begin{equation}\label{eqn:update}
		H_{k+1} = H_k + \alpha_k s_k s_k^T - \frac{\alpha_k}{\gamma_k^2}(H_ky_k + \alpha_k (s_k^Ty_k)s_k)(H_ky_k + \alpha_k (s_k^Ty_k)s_k)^T,
	\end{equation}
	where
	\begin{equation*}
		\gamma_k = 0.5 + \sqrt{0.25 + \alpha_k y_k^TH_ky_k + \alpha_k^2 (s_k^Ty_k)^2}.
	\end{equation*}
\end{theorem}

\begin{proof}
	We begin by showing that the objective function $\Upsilon(B)$ of (\ref{eqn:soft_QN_problem}) is strictly convex.
	Since $\tr(BH_k)-\log(\det(BH_k))$ has been proven to be strictly convex on the cone of positive definite matrices in \cite{Fletrcher1991Variational}, we only need to show that the penalty term
	\begin{equation*}
		\alpha_k \|s_k - B^{-1}y_k\|_B^2 = \alpha_k( s_k^T B s_k - 2 s_k^Ty_k + y_k^T B^{-1}y_k),
	\end{equation*}
	is convex.  We see that the first term is linear in $B$ and therefore convex, while the second is constant. Regarding the third term, taking the inverse of a positive definite matrix is a convex operation \cite{NORDSTROM20111489} in the sense that for any positive definite matrices $X$ and $Y$, and scalar $\theta \in [0,1]$, it holds that
	\begin{equation*}
		(\theta  X + (1 - \theta)Y)^{-1} \preceq \theta  X^{-1} + (1 - \theta )Y^{-1}.
	\end{equation*}
	From this, it follows that 
	\begin{equation*}
		y_k^T (\theta X + (1 - \theta )Y)^{-1} y_k \leq \theta y_k^T X^{-1}y_k + (1 - \theta )y_k^TY^{-1}y_k,
	\end{equation*}
	showing that the third term is convex too. 
	As a result, $\Upsilon(B)$ is strictly convex on the cone of positive definite matrices, and the optimal solution to (\ref{eqn:soft_QN_problem}) is unique and must satisfy the first-order optimality condition
	\begin{equation}\label{eqn:criticalpoint}
		\nabla \Upsilon(B) = H_k - B^{-1} + \alpha_k(s_k s_k^T - B^{-1} y_k y_k^T B^{-1}) = 0.
	\end{equation}
	Since the difference between $B^{-1}$ and $H_k$ is a symmetric rank-2-matrix, it is natural to make the following ansatz: 
	\begin{equation*}
		B^{-1} = H_k + \alpha_k (s_k s_k^T - u_k u_k^T).
	\end{equation*}
	From Equation~\eqref{eqn:criticalpoint}, we can see that the above equality holds if $u_k = B^{-1}y_k$ or $u_k = -B^{-1}y_k$. Without loss of generality, we choose 
	\begin{equation}\label{eqn:uk2}
		u_k = B^{-1}y_k = H_k y_k + \alpha_k (s_k^Ty_k) s_k - \alpha_k (u_k^Ty_k) u_k.
	\end{equation}
	To determine $u_k$, we first determine the unknown quantity 
	\begin{equation*}
		u_k^Ty_k = y_k^T u_k = y_k^T H_k y_k + \alpha_k (s_k^Ty_k)^2 - \alpha_k (u_k^Ty_k)^2.
	\end{equation*}
	Solving for $u_k^Ty_k$ yields
	\begin{equation*}
		u_k^Ty_k = -\frac{1}{2\alpha_k} \pm \sqrt{\frac{1}{4 \alpha_k^2}+\frac{y_k^TH_ky_k}{\alpha_k}+(s_k^Ty_k)^2}.
	\end{equation*}
	There are two possible solutions, but for $B^{-1}$ to be positive definite, we need that $y_k^TB^{-1}y_k = u_k^T y_k > 0$, so we use the positive solution. By Equation~\eqref{eqn:uk2},
	\begin{equation*}
		u_k = \frac{1}{1+\alpha_k u_k^Ty_k} (H_ky_k + \alpha_k (s_k^Ty_k)s_k),
	\end{equation*}
	so plugging in the obtained value of $u_k^Ty_k$ gives 
	\begin{equation*}
		u_k = \frac{H_ky_k + \alpha_k (s_k^Ty_k)s_k}{0.5 + \sqrt{0.25 + \alpha_k y_k^TH_ky_k+\alpha_k^2(s_k^Ty_k)^2}}.
	\end{equation*}
	Thus, setting $H_{k+1}$ to the value of $B^{-1}$ at the critical point leads to the update formula
	\begin{equation}
		H_{k+1} = H_k + \alpha_k s_k s_k^T - \frac{\alpha_k}{\gamma_k^2}(H_ky_k + \alpha_k (s_k^Ty_k)s_k)(H_ky_k + \alpha_k (s_k^Ty_k)s_k)^T 
	\end{equation}
	with 
	\begin{equation*}
		\gamma_k = 0.5 + \sqrt{0.25 + \alpha_k y_k^TH_ky_k + \alpha_k^2 (s_k^Ty_k)^2}.
	\end{equation*}
	
	It only remains to validate that $H_{k+1}$ is positive definite if $H_k$ is. To this end, we first
	observe that $H_{k+1}$ is the Schur complement of $\frac{\gamma_k^2}{\alpha_k}$ in the matrix
	\begin{equation*}
		S := \begin{bmatrix} H_k + \alpha_k s_k s_k^T & (H_k + \alpha_k s_k s_k^T)y_k \\ y_k^T(H_k + \alpha_k s_k s_k^T) & \frac{\gamma_k^2}{\alpha_k} \end{bmatrix}.
	\end{equation*}
	Since $\frac{\gamma_k^2}{\alpha_k}$ is positive, $H_{k+1}$ is positive definite if and only if $S$ is. Furthermore, $S$ is positive definite if and only if $H_k+\alpha_k s_ks_k^T$ and its Schur complement in $S$ is. But $H_k+\alpha_k s_ks_k^T$ must be positive definite since $H_k$ is, so it remains to check whether its Schur complement is positive. The Schur complement is 
	\begin{align*}
		&\frac{\gamma_k^2}{\alpha_k} - y_k^T (H_k+\alpha_k s_ks_k^T) (H_k+\alpha_k s_ks_k^T)^{-1}(H_k+\alpha_k s_ks_k^T) y_k \\
		=& \frac{0.5 + \alpha_k y_k^TH_ky_k + \alpha_k^2 (s_k^Ty_k)^2 + \sqrt{0.25 + \alpha_k y_k^T H_ky_k +\alpha_k^2 (s_k^Ty_k)^2}}{\alpha_k} \\
		& - y_k^T (H_k+\alpha_k s_ks_k^T)y_k \\
		=& \frac{1}{2 \alpha_k} + \sqrt{\frac{1}{4 \alpha_k^2} + \frac{y_k^T H_ky_k}{\alpha_k} + (s_k^Ty_k)^2},
	\end{align*}
	which is indeed positive. 
	The proof is complete.
\end{proof}

\subsection{Properties of the soft QN update}

We will now establish four attractive properties of the soft QN update: it always generates descent directions, it recovers BFGS in the limit, it treats positive and negative curvature equally, and it is scale invariant. Some of these properties are immediate, while others require more detailed proofs and arguments. Nevertheless, we believe that they are all essential features that contribute to the efficacy of the method. 

\subsubsection{Soft QN always generates descent directions}
A remarkable feature of the soft QN update is that it guarantees that $H_k$ remains positive definite for any $\alpha_k>0$, even if $s_k^Ty_k\leq 0$. As mentioned in Section~\ref{sec:QN_methods}, this guarantees that $-H_k \nabla f(x_k)$ is a descent direction, and that it is the unique minimizer of the quadratic model function used to approximate the actual loss function. 

\subsubsection{Soft QN recovers BFGS in the limit}\label{sec:limit_of_soft_QN}
The proposed soft QN method is derived as the solution to a matrix optimization problem that is a relaxed version of the one solved by BFGS. It is therefore natural to expect that the soft QN update converges to the BFGS update when $\alpha_k \to \infty$, provided that
the BFGS update preserves positive definiteness. The next theorem confirms that this is indeed the case.
\begin{theorem} \label{thm:recovery}
	If $s_k^Ty_k > 0$, the right hand side of Equation~\eqref{eqn:update} converges to 
	\begin{align*}
		\left (I - \frac{s_ky_k^T}{s_k^Ty_k} \right)H_k\left (I - \frac{y_k s_k^T}{s_k^Ty_k} \right) + \frac{s_ks_k^T}{s_k^Ty_k}
	\end{align*}
	as $\alpha_k \to \infty$.
\end{theorem}
\begin{proof}
	Equation~\eqref{eqn:update} can be rewritten as
	\begin{align*}
		H_{k+1} =& H_k + \alpha_k s_k s_k^T \\
		&- \frac{\alpha_k}{\gamma_k^2}(H_ky_ky_k^TH_k + \alpha_k (s_k^Ty_k)(H_ky_ks_k^T+s_ky_k^TH_k)+\alpha_k^2(s_k^Ty_k)^2s_ks_k^T) \\
		= & \left (I - \frac{\alpha_k^2 s_k^Ty_k}{\gamma_k^2}s_ky_k^T \right)H_k \left (I - \frac{\alpha_k^2 s_k^Ty_k}{\gamma_k^2}y_ks_k^T \right) \\ 
		&+ \alpha_k \left(1 - \frac{\alpha_k^2(s_k^Ty_k)^2}{\gamma_k^2} - \frac{\alpha_k^3 (s_k^Ty_k)^2y_k^TH_ky_k}{\gamma_k^4}\right)s_k s_k^T - \frac{\alpha_k}{\gamma_k^2}H_ky_ky_k^TH_k
	\end{align*}
	and should be compared to the BFGS update:
	\begin{align*}
		H_{k+1} = \left (I - \frac{1}{s_k^Ty_k}s_ky_k^T \right)H_k\left (I - \frac{1}{s_k^Ty_k}y_k s_k^T \right) + \frac{1}{s_k^Ty_k}s_ks_k^T.
	\end{align*}
	To show that the former update converges to the latter, we will show that the first two terms in our update converge to the corresponding terms in the BFGS update, while the third term vanishes.  
	We first show that $\alpha_k^2s_k^Ty_k/\gamma_k^2$ converges to $1/s_k^Ty_k$. Since 
	\begin{equation*}
		\gamma_k^2 = 0.5 + \alpha_k y_k^TH_ky_k + \alpha_k^2 (s_k^Ty_k)^2 + \sqrt{0.25 + \alpha_k y_k^T H_ky_k +\alpha_k^2 (s_k^Ty_k)^2},
	\end{equation*}
	the dominant term in the denominator becomes $\alpha_k^2 (s_k^Ty_k)^2$ as $\alpha_k \to \infty$, and therefore, 
	\begin{equation*}
		\lim_{\alpha_k \to \infty} \frac{\alpha_k^2s_k^Ty_k}{\gamma_k^2} = \lim_{\alpha_k \to \infty} \frac{\alpha_k^2s_k^Ty_k}{\alpha_k^2(s_k^Ty_k)^2} = \frac{1}{s_k^Ty_k}.
	\end{equation*}
	Hence, the first term in our update converges to the first term in the BFGS update. 
	We also see that the third term in our update, $- \frac{\alpha_k}{\gamma_k^2}H_ky_ky_k^TH_k$, vanishes as $\alpha_k \to \infty$ since the numerator grows proportionally to $\alpha_k$ while the denominator grows as $\alpha_k^2$. It remains to show that the second term in our update converges to the second term in the BFGS update, \emph{i.e.} that
	\begin{equation*}
		\lim_{\alpha_k \to \infty} \alpha_k \left(1 - \frac{\alpha_k^2(s_k^Ty_k)^2}{\gamma_k^2} - \frac{\alpha_k^3 (s_k^Ty_k)^2y_k^TH_ky_k}{\gamma_k^4}\right) = \frac{1}{s_k^Ty_k}.
	\end{equation*}
	We have that 
	\begin{align*}
		& \alpha_k \left(1 - \frac{\alpha_k^2(s_k^Ty_k)^2}{\gamma_k^2} - \frac{\alpha_k^3 (s_k^Ty_k)^2y_k^TH_ky_k}{\gamma_k^4}\right) \\
		=&  \alpha_k \Bigg (\frac{ 0.5 + \alpha_k y_k^TH_ky_k + \sqrt{0.25 + \alpha_k y_k^T H_ky_k +\alpha_k^2 (s_k^Ty_k)^2}}{\gamma_k^2}  - \frac{\alpha_k^3 (s_k^Ty_k)^2y_k^TH_ky_k}{\gamma_k^4} \Bigg ) \\
		=& \frac{\alpha_k}{\gamma_k^2} \left ( 0.5 + \sqrt{0.25 + \alpha_k y_k^T H_ky_k + \alpha_k^2 (s_k^Ty_k)^2} \right) + \frac{\alpha_k^2}{\gamma_k^2} \left( y_k^T H_k y_k \left(1 - \frac{(s_k^Ty_k)^2\alpha_k^2}{\gamma_k^2} \right) \right).
	\end{align*}
	When $\alpha_k \to \infty$, the term $\alpha_k^2 (s_k^Ty_k)$ will be the dominant term under the square root, and the first term will approach $\alpha_k^2 |s_k^Ty_k|/\gamma_k^2$. When $s_k^Ty_k > 0$, this is $\alpha_k^2 s_k^Ty_k/\gamma_k^2$ which, as we have already shown, converges to $1/s_k^Ty_k$ as $\alpha_k \to \infty$. The second term will converge to 0, as $\alpha_k^2/\gamma_k^2$ converges to $1/(s_k^Ty_k)^2$ and $1 - (s_k^Ty_k)^2\alpha_k^2/\gamma_k^2$ goes to 0. Thus, we recover the BFGS update formula as $\alpha_k\rightarrow \infty$.
\end{proof}

\subsubsection{Soft QN treats positive and negative curvature equally}

If we repeat the argument in Theorem~\ref{thm:recovery} for  $s_k^Ty_k < 0$, we find that soft QN tends to a BFGS update where $y_k$ is replaced by $-y_k$. Thus, 
the limiting behavior of the soft QN update as $\alpha_k\rightarrow \infty$ is a BFGS
update with a reversed sign for $y_k$ whenever $s_k^Ty_k <0$. One can see that soft QN treats the curvature as if it was positive instead of negative in this limiting case. This is no coincidence: the soft QN update formula \eqref{eqn:update} does not depend on the sign of $y_k$ or $s_k$, \emph{i.e.} replacing either $y_k$ by $-y_k$ or $s_k$ by $-s_k$ would result in the same update. This means that any quasi-Newton algorithm based on this update is only concerned with the absolute value of the curvature in the direction $s_k$ and not the sign, and will re-scale the gradient to take longer steps in directions with lower absolute curvature. 

To see why this could be advantageous, consider the argument laid out in \cite{10.5555/2969033.2969154} to motivate a modification of Newton's method from trust-region principles. The authors consider an iterative algorithm $x_{k+1} = x_k + p_k$ where $p_k$ is the solution to the trust-region problem
\begin{equation}\label{eqn:Newton_tr}
	\begin{aligned}
		& \underset{p}{\text{minimize}} & f(x_k) + p^T \nabla f(x_k) \\
		&\text{ subject to} & p^T | \nabla^2 f(x_k) | p \leq c.
	\end{aligned}
\end{equation}
Here, $|\nabla ^2 f(x_k)|$ denotes the matrix obtained by taking the Hessian evaluated at $x_k$ and setting all its eigenvalues to their absolute values while retaining its eigenvectors, while $c>0$ is the trust-region radius. As usual for trust-region problems, the objective function is replaced by a model based on its Taylor expansion, although trust-region problems typically used both first- and second-order terms. 
An upper bound of the second-order term is instead used to determine the shape of the trust-region. In contrast to trust-region methods that just use a 2-norm constraint, this formulation takes into account that the first-order model becomes inaccurate quicker in directions of high curvature, whether positive or negative. Solving this trust-region update gives that $p_k$ is proportional to $- |\nabla^2 f(x_k)|^{-1} \nabla f(x_k)$. The method proposed in \cite{10.5555/2969033.2969154}, called saddle-free Newton, is based on this insight. A similar idea is the basis for the non-convex Newton method proposed in \cite{Paternain2019NewtonNonconvex}. While the practical methods proposed in these works make additional modifications to the iteration $x_{k+1} = x_k - |\nabla^2 f(x_k)|^{-1} \nabla f(x_k)$ to be scalable in higher dimensions, both works argue that it is advantageous to move further in directions of low absolute curvature in order to escape from saddle points.

While both saddle-free Newton and non-convex Newton, in contrast to the regular Newton method, escape saddle points by moving far in directions corresponding to eigenvalues with a low absolute value, they are unable to detect directions of low curvature that are linear combinations of eigenvectors with eigenvalues of opposite sign. Following the trust-region principle, it would be desirable for a method to identify such directions and re-scale the gradient accordingly to obtain a suitable step. Here, a quasi-Newton method that updates its Hessian approximation based on the absolute value of the curvature can have an advantage. To illustrate this, we consider the 2-dimensional toy-example of optimizing
\begin{align*}
	f(x,y) &=
	(x-0.7)^2 ( (x+0.7)^2+0.1) + (y+0.7)^2 ( (y-0.7)^2+0.1).
\end{align*} 

\begin{figure}
	\centering
	\includegraphics[width = 0.9\textwidth]{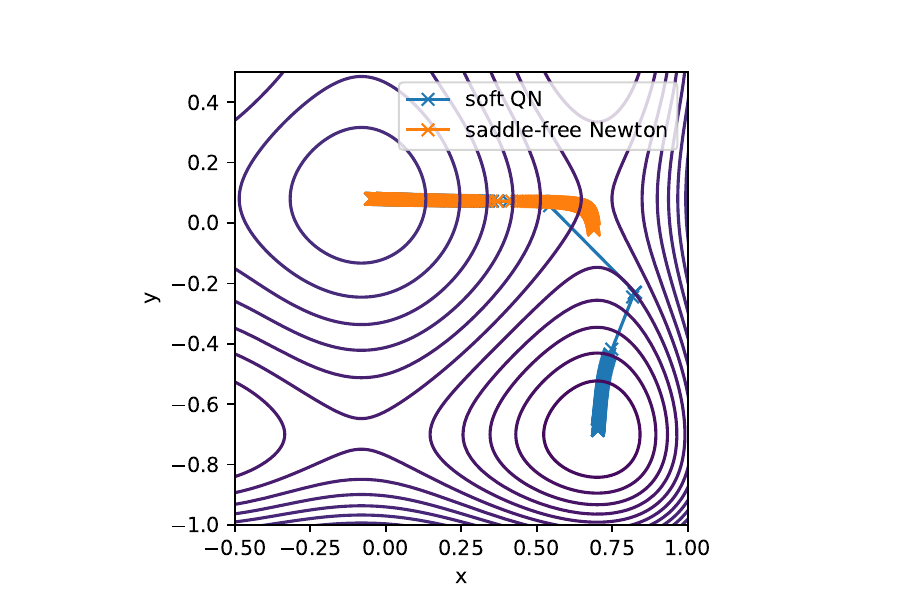}
	\caption{~The soft QN method detects a diagonal direction of low curvature and takes a large step along it. The saddle-free Newton method is unable to perceive the low curvature in that direction and moves further toward the saddle point before changing its course.}
	\label{fig:toy_example}
\end{figure}
This function has one local maximum, four saddle points and four local minima, with one of them also being a global minimum. The function is separable in the two variables, so the Hessian is diagonal. In Figure~\ref{fig:toy_example}, we show plots of 500 iterations of $x_{k+1} = x_k - \eta |\nabla^2 f(x_k)|^{-1} \nabla f(x_k)$ (denoted as saddle-free Newton), against iterations of  $x_{k+1} = x_k - \eta H_k \nabla f(x_k)$ where $H_k$ is updated according to Equation~\eqref{eqn:update} and $H_0$ is $|\nabla^2 f(x_0)|^{-1}$. We set $ x_0 = (-0.05,0.08)$, the step-size $\eta = 0.01$ and the penalty parameter $\alpha_k = 8 \cdot 10^5$.  The parameters have been chosen to highlight the behavior that we want to illustrate. While the step-size of Newton-type methods is typically set to 1 or determined by line-search, the short step-size and high penalty parameter ensure that the methods will take long steps if and only if they discover a direction of low curvature. 
%
%
This phenomenon is illustrated in Figure~\ref{fig:toy_example}: at the point $(0.543,0.0574)$, the Hessian has the values $1.78$ and $-1.72$ on its diagonal. The eigenvalues are of opposite signs but very close in magnitude, meaning that saddle-free Newton treats the Hessian almost as a scaled identity matrix. Soft QN, on the contrary, detects that the curvature is close to 0 for a diagonal direction, takes a long step in that direction to end up at $(0.827,-0.230)$ in the next iteration. This sets it far ahead of saddle-free Newton and allows it to find the local minimum within 500 iterations.

\subsubsection{Soft QN is scale invariant}

Scale invariance, \emph{i.e.} invariance under linear transformations of the variables, is a desirable property of optimization algorithms since it guarantees that the algorithm is insensitive to ill-conditioning of the Hessian. To understand why, consider a strongly convex problem with an ill-conditioned Hessian.  
By a 
linear transformation of variables, we can transform such a problem into an equivalent one with an arbitrary positive definite Hessian. If the algorithm is scale invariant, then the iterations will not be affected by the variable transformation, and the method will perform equally well on the original (ill-conditioned) problem formulation as the transformed (well-conditioned) one. The next result shows that a scale-invariance property holds for the soft QN update.
\begin{proposition}
	The soft QN update is scale-invariant, i.e., under a change of variables  $\tilde{x}_k = A x_k$, where $A$ is an invertible square matrix, $\tilde{H}_k = A H_k A^T$ implies that $\tilde{H}_{k+1} = A H_{k+1} A^T$, where $\tilde{H}_k$ is the Hessian inverse approximation associated with the variables $\tilde{x}_k.$ 
\end{proposition}
\begin{proof}
	Defining $\tilde{y}_k$ and $\tilde{s}_k$ from $\tilde{x}_k$ analogously to how $s_k$ and $y_k$ are defined from $x_k$, we have that $\tilde{s}_k = A s_k$ and $\tilde{y}_k = A^{-T} y_k$. Since $\tilde{y_k}^T \tilde{s}_k = y_k A^{-1} A s_k = y_k^T s_k$ and $\tilde{y}_k^T \tilde{H}_k \tilde{y}_k = y_k A^{-1} A H_k A^T A^{-T}y_k = y_k^T H_k y_k$, so the quantity $\gamma_k$ from Equation~\eqref{eqn:update} is scale invariant. Therefore, 
	
	\begin{equation*}
		\begin{aligned}
			\tilde{H}_{k+1}  =& \tilde{H}_k + \alpha_k \tilde{s}_k \tilde{s}_k^T - \frac{\alpha_k}{\gamma_k^2}(\tilde{H}_k \tilde{y}_k + \alpha_k (\tilde{s}_k^T \tilde{y}_k)\tilde{s}_k)(\tilde{H}_k \tilde{y}_k + \alpha_k (\tilde{s}_k^T\tilde{y}_k)\tilde{s}_k)^T \\
			=& A H_k A^T + \alpha_k A s_k s_k^T A^T - \\ & \frac{\alpha_k}{\gamma_k^2} ( A H_k A^{T} A^{-T} y_k + \alpha_k (s_k^Ty_k) A s_k)( A H_k A^{T} A^{-T} y_k + \alpha_k (s_k^Ty_k) A s_k)^T \\
			=& A (H_k + \alpha_k s_k s_k^T - \alpha_k (H_k y_k + \alpha_k (s_k^Ty_k) s_k)(H_k y_k + \alpha_k (s_k^Ty_k) s_k)^T ) A^T \\
			=& A H_{k+1} A^T.
		\end{aligned}
	\end{equation*}
\end{proof}

The preservation of scale-invariance is a property that soft QN shares with all quasi-Newton methods in the Broyden class \cite{NoceWrig06}. However, if $H_0$ is not scale-invariant  then the subsequent  Hessian inverse approximations will not be scale-invariant either.  While it is possible to choose a scale-invariant $H_0$, e.g. setting it equal to the true Hessian at the initial point, this is typically computationally expensive and it is common to simply set $H_0$ to the identity matrix, which corresponds to $\tilde{H}_0 = A A^T$ under the change of variables $\tilde{x}_k = A x_k$. However, as the iteration number increases, the influence of the initial matrix diminishes, and one can argue that the  Hessian inverse approximation becomes ``more scale invariant" and thereby more resistant to ill-conditioning. This is confirmed empirically with traditional quasi-Newton methods outperforming gradient descent on ill-conditioned problems \cite{NoceWrig06}.

\subsection{Comparison with SP-BFGS}

The SP-BFGS method~\cite{Irwin2023} is similar to soft QN, but it is based on solving a relaxed version of problem~\eqref{eqn:BFGS_problem} instead of problem~\eqref{eqn:BFGS_problem_2}. Specifically, it is derived as the solution to
\begin{equation}\label{eqn:SP-BFGS problem}
	\begin{aligned}
		& \underset{H}{\text{minimize}}
		& & \|W_k^{(1/2)} (H - H_k) W_k^{(1/2)}\|_F^2 + \frac{\beta_k}{2} \|W_k^{(1/2)} (Hy_k - s_k)\|_2^2 \\
		& \text{subject to} & & H = H^T
	\end{aligned}
\end{equation}
for some strictly positive penalty parameter $\beta_k>0$. In this formulation, the inverse secant condition is weighted by the (principal square root of the) weight matrix $W_k$, which can be any positive definite matrix that satisfies the secant condition. This weighting ensures scale-invariance. For any such $W_k$, the solution of this problem leads to the update formula 
\begin{equation}\label{eqn:SP-BFGS update}
	H_{k+1} = \left(I-\omega_k s_k y_k^T \right) H_k \left(I - \omega_k y_k s_k^T \right) + \omega_k \left( \dfrac{\pi_k}{\omega_k} + (\pi_k- \omega_k) y_k^T H_k y_k \right) s_k s_k^T,
\end{equation}
where 
\begin{equation*}
	\pi_k = \dfrac{1}{s_k^Ty_k + \frac{1}{\beta_k}} \ , \ \omega_k = \dfrac{1}{s_k^T y_k + \frac{2}{\beta_k}}.
\end{equation*}
By Lemma 1 in~\cite{Irwin2023}, the update \eqref{eqn:SP-BFGS update} preserves positive definiteness when 
$s_k^Ty_k > -\beta_k^{-1}$. This means that it can still yield positive definite matrices when $s_k^T y_k < 0$, as long as the penalty parameter $\beta_k$ is not too large, and the update formula can be used as long as care is taken to choose $\beta_k$ appropriately. 
However, as discussed in Section~\ref{sec:QN_methods}, there is no matrix $W_k$ that satisfies the secant condition when $s_k^T y_k < 0$, so the matrix optimization problem~\eqref{eqn:SP-BFGS problem} is not justified and it is not clear in what sense  $H_{k+1}$ is an optimal update of $H_k$. In contrast, $H_{k+1}$  given by the proposed update formula \eqref{eqn:update} is positive definite for any value of $\alpha_k>0$, and it always exists as an optimal solution to the problem \eqref{eqn:soft_QN_problem}.

When the penalty parameters of soft QN and SP-BFGS approach infinity, the behavior of the update formulas is similar, but with one important difference: While the SP-BFGS update always  converges to the BFGS update, soft QN converges to a modified update where the sign of $y_k$ is changed if $s_k^T y_k$ is negative (cf. the discussion in Section~\ref{sec:limit_of_soft_QN}).  

\section{Convergence}\label{sec:Convergence}

In this section, we demonstrate how the soft QN update can be integrated into a quasi-Newton method with guaranteed convergence rate for strongly convex problems also in the presence of bounded errors in gradient and function evaluations. The integration relies on a uniform bound on the Hessian inverse approximation generated by soft QN. 
\subsection{Bounding the  Hessian inverse approximation}\label{sec:bound}
In this subsection, we derive a method to bound $H_k$ in a way that takes into account its size in the current iteration, allowing larger changes to the matrix if it is far from a desired upper or lower bound. We do so by having a different penalty parameter $\alpha_k$ in each iteration $k$ and choosing it based on upper and lower bounds that depend on a size measure of $H_k$. Several convergence proofs, e.g.~\cite{wang2017stochastic}, ~\cite{moritz2016linearly}, and~\cite{Irwin2023},  rely on upper and lower bounds on the eigenvalues of $H_k$, and we will do that here too. We denote the maximum eigenvalue of a matrix $A$ with real-valued eigenvalues by $\lambda_{\max}(A)$ and the minimum eigenvalue by $\lambda_{\min}(A)$.

\begin{lemma}\label{lemma:1}
	Given $\psi I \preceq H_{0} \preceq \Psi I$ with some constants $0<\psi\le \Psi <\infty$, if the penalty parameter $\alpha_k$ satisfies 
	\begin{equation}\label{eq:alpha}
		\alpha_k \leq \min\left\{  \frac{\lambda_{\min}(H_k) - \psi}{(\|s_k\|_2 + \|H_ky_k\|_2)^2}, ~ \dfrac{\Psi-\lambda_{\max}(H_k)}{\|s_k\|_2^2} \right\}, 
	\end{equation}
	then the sequence $\{H_k\}_k$ generated by the proposed soft QN satisfies 
	\begin{equation}
		\psi I   \preceq H_k \preceq \Psi I, ~\forall k. 
	\end{equation}
\end{lemma}

\begin{proof}
	From Equation~\eqref{eqn:update}, we have
	\begin{equation*}
		\lambda_{\max}(H_{k+1}) \leq \lambda_{\max}(H_k) + \alpha_k \|s_k\|_2^2.
	\end{equation*}
	Hence, we can ensure that the eigenvalues are smaller than some constant $\Psi$ by letting 
	\begin{equation*}
		\alpha_k \leq \dfrac{\Psi-\lambda_{\max}(H_k)}{\|s_k\|_2^2}.
	\end{equation*}
	To lower bound $H_k$, we observe from Equation~\eqref{eqn:update} that 
	\begin{equation*}
		\lambda_{\min}(H_{k+1}) \geq \lambda_{\min}(H_k) - \alpha_k \|u_k\|_2^2
	\end{equation*}
	with
	\begin{equation*}
		u_k = \frac{H_ky_k + \alpha_k (s_k^Ty_k)s_k}{0.5 + \sqrt{0.25 + \alpha_k y_k^TH_ky_k+\alpha_k^2(s_k^Ty_k)^2}}.
	\end{equation*}
	Since $u_k$ depends on $\alpha_k$, we want to eliminate $\|u_k\|_2$ from the above expression to derive a practical bound. We can do so by observing that
	\begin{align*}
		\|u_k\|_2 & = \frac{\|H_ky_k + \alpha_k (s_k^Ty_k)s_k\|_2}{0.5 + \sqrt{0.25 + \alpha_k y_k^TH_ky_k+\alpha_k^2(s_k^Ty_k)^2}} \\ & \leq \frac{\|H_ky_k \|_2 + \alpha_k \|(s_k^Ty_k)s_k\|_2}{0.5 + \sqrt{0.25 + \alpha_k y_k^TH_ky_k+\alpha_k^2(s_k^Ty_k)^2}} \\ & \leq \frac{\|H_k y_k\|_2}{1} + \frac{\alpha_k |s_k^Ty_k| \|s_k||_2}{\sqrt{\alpha_k^2 (s_k^Ty_k)^2}} = \|H_ky_k\|_2 + \|s_k\|_2.
	\end{align*}
	With this bound, $\lambda_{\min}(H_{k+1})$ is guaranteed to stay above $\psi$ if 
	\begin{equation*}
		\alpha_k \leq \frac{\lambda_{\min}(H_k) - \psi}{(\|s_k\|_2 + \|H_ky_k\|_2)^2}. 
	\end{equation*}
	This completes the proof. 
\end{proof}

While Lemma~\ref{lemma:1} shows that $H_k$ will remain bounded if $\alpha_k$ is small enough, the computational demands are something to consider if we want to use it in practice. While $\lambda_{\min}(H_k)$ and $\lambda_{\max}(H_k)$ can be computed with power iteration, doing so might be slow if there are several eigenvalues that are close in magnitude. However, note that one can replace the maximum eigenvalue of $H_k$ by an upper bound and the minimum eigenvalue by a lower bound and still ensure boundedness of $H_k$. There are several ways that such bounds can be calculated \cite{TrDetBounds} \cite{WOLKOWICZ1980471}, but to ensure that computing it does not overshadow other operations performed in the algorithm, e.g. the multiplication of $H_k$ with the gradient, we want the evaluation of the bound to scale at most as $O(n^2)$, where $n$ is the number of variables in the problem. An upper bound that satisfies this and works for any $n \times n$ complex matrix with real eigenvalues $A$ is

\begin{equation*}
	\lambda_{\max}(A) \leq \frac{\tr(A)}{n} + \sqrt{\dfrac{\frac{\tr(A^2)}{n}-\big (\frac{\tr(A)}{n} \big )^2}{n-1}}.
\end{equation*}
An inexpensive way of enforcing a lower bound of the lowest eigenvalue is to add a small bias term $\lambda I$ to the nominal $H_k$ before multiplying it with the gradient. Since $H_k$ is guaranteed to be positive definite no matter the values of $s_k$ and $y_k$, $\lambda_{\min}(H_k + \lambda I) \geq \lambda$. 

\subsection{Convergence guarantees}
With the uniform bound on the  Hessian inverse approximation established in Lemma \ref{lemma:1}, it is safe to use the  Hessian inverse approximations generated by soft QN to construct approximate Newton directions for solving the minimization problem \eqref{eq:objective} in the presence of general bounded noise.  
More specifically, assume that we want to optimize a function $\phi(x)$ but only have access to corrupted measurements of the function and the gradient 
\begin{equation}\label{eq:noisy-loss-g}
	f(x) := \phi(x) + n_f(x), ~g(x) := \nabla \phi(x) + n_g(x),
\end{equation}
where $n_f(x) $ and $n_g(x)$ are noises in the function and gradient, respectively. 
We make the following assumption on the noises.
\begin{assumption}\label{asm:bounded noise}
	There exist positive constants $e_f$ and $e_g$ such that
	$$|n_f(x)| \leq e_f, ~ \| n_g(x) \|_2 \leq e_g,\quad \forall x.$$ 
\end{assumption}
The soft QN algorithm for solving problem~\eqref{eq:objective} is summarized in Algorithm~\ref{alg}. 

\begin{algorithm}
	\caption{ Soft QN Algorithm}
	\label{alg}
	\begin{algorithmic}[1]
		\State {\bf Set} $\eta_k$, $H_0\succ 0$, $x_0$, and $g_0$
		\For{$k$=$0$, $1$, \ldots, $K-1$}
		\State Update $p_k =-H_k g_k$
		\State Update $x_{k+1} = x_k+\eta_k p_k$
		\State Update $g_{k+1} = g(x_{k+1})$
		\State Update $s_k = x_{k+1}-x_k$
		\State Update $y_k = g_{k+1}-g_k$
		\State Choose penalty parameter $\alpha_k$ satisfying \eqref{eq:alpha}
		\State Update $\gamma_k = 0.5 + \sqrt{0.25 + \alpha_k y_k^TH_ky_k + \alpha_k^2 (s_k^Ty_k)^2}$
		\State Update
		$$    H_{k+1} = H_k + \alpha_k s_k s_k^T - \frac{\alpha_k}{\gamma_k^2}(H_ky_k + \alpha_k (s_k^Ty_k)s_k)(H_ky_k + \alpha_k (s_k^Ty_k)s_k)^T $$
		\EndFor
		\State {\bf Output} $x_K$ 
	\end{algorithmic}
\end{algorithm}

For twice continuously differentiable, $L$-smooth and $\mu$-strongly convex functions, with bounded noise on gradient and function evaluations, the authors of \cite{Irwin2023} show that general quasi-Newton methods with bounded  Hessian inverse approximations converge linearly to a neighborhood of the optimal solution when the step-size is fixed or is generated by line-search. As these theorems are also applicable here, we state them and discuss how they apply to our quasi-Newton method.

\begin{assumption}\label{asm:lips}
	The gradient 
	$\nabla \phi $ is  Lipschitz continuous with Lipschitz constant $L$; namely for any $x, y \in \mathbb{R}^n$,
	$$
	\|\nabla \phi(x)-\nabla \phi(y)\| \leq L\|x-y\|.
	$$
\end{assumption}
\begin{assumption}\label{asm:strong-cvx}
	The function $\phi $ is  $\mu$-strongly convex, i.e.,  there exists a positive constant $\mu$ such that for all $x, y \in \mathbb{R}^n$:
	\begin{equation*}
		\phi(y)\ge \phi(x)+ \nabla \phi(x)^T(y-x)+\frac{\mu}{2}\|x-y\|^2. 
	\end{equation*}
\end{assumption}

The following theorem provides convergence guaranteed under a fixed step-size $\eta_k=\eta$. 
\begin{theorem}{(Fixed step-size $\eta_k=\eta$)} \label{thm: rate}
	Consider Algorithm \ref{alg}. Under Assumptions \ref{asm:bounded noise}--\ref{asm:strong-cvx}, if the step-size $\eta$ satisfies
	$$
	0<\eta \leq \frac{\varepsilon}{L \Psi^2}<\frac{\psi}{L \Psi}\left[\frac{A}{((1+A) \Psi+\psi)}\right], 
	$$
	where $0<\varepsilon<\frac{A \psi \Psi}{((1+A) \Psi+\psi)}$ and $A>0$ is some constant, 
	then for any  $x_k \notin \mathcal{N}_1\left(\psi, \Psi, A \Psi e_g\right)$, where
	\begin{equation*}
		\begin{aligned}
			\mathcal{N}_1\left(\psi, \Psi, A \Psi e_g\right) = \left\{x \mid \phi(x) \leq \phi^{\star}+\frac{1}{2 \mu}\left(\frac{(1+A) \Psi e_g}{\psi}\right)^2\right\},    
		\end{aligned}   
	\end{equation*}
	we have
	$$
	\begin{aligned}
		&\phi(x_{k+1})-\left(\phi(x^{\star})+\frac{1}{2 \mu}\left(\frac{(1+A) \Psi e_g}{\psi}\right)^2\right) \\
		\leq&(1-\eta\varepsilon \mu)\left(\phi(x_k)-\left(\phi(x^{\star})+\frac{1}{2 \mu}\left(\frac{(1+A) \Psi e_g}{\psi}\right)^2\right)\right).   
	\end{aligned}
	$$
\end{theorem}
\begin{proof}
	See \cite[Theorem 4]{Irwin2023}. 
\end{proof}

Besides a fixed step-size, we can also use line-search to generate an $\eta_k$ for iteration $k$. 
To account for the noise in the function evaluations, we use the relaxed Armijo condition introduced in~\cite{Irwin2023}
\begin{equation}\label{eq:Armijo}
	f(x_k + \eta_k p_k) \le f(x_k) + \eta_k c p_k^T g(x_k) + 2 \varepsilon_{\rm tol}.
\end{equation}
Here, $c>0$, $\varepsilon_{\rm tol}> 0 $, and the step-size $\eta_k$ is the maximum value in $\left\{\tau^j \eta_0 \mid j=0,1,2, \ldots\right\}$ that satisfies the relaxed Armijo condition \eqref{eq:Armijo} for some $\eta_0>0$  and $\tau\in (0,1) $. When we use such step-sizes $\eta_k$ in Algorithm \ref{alg} we get a soft QN Algorithm with line-search that has the following guarantees. 
\begin{theorem}(Line-search step-size $\eta_k$ satisfying \eqref{eq:Armijo})\label{thm:linesearch}
	Consider Algorithm \ref{alg}. Under Assumptions \ref{asm:bounded noise}--\ref{asm:strong-cvx}, let $\varepsilon_{\rm tol}>\epsilon_f$ and
	$$
	0<c \leq \frac{\varepsilon}{2 \Psi} \frac{\left(1-\frac{\psi}{(1+A) \Psi}\right)}{\left(1+\frac{\psi}{(1+A) \Psi}\right)}<\frac{\psi}{2} \frac{A((1+A) \Psi-\psi)}{((1+A) \Psi+\psi)^2},
	$$
	where $0<\varepsilon<\frac{A \psi \Psi}{((1+A) \Psi+\psi)}$ and $A>0$ is some constant.
	Then, for any $x_k \notin \mathcal{N}_1\left(\psi, \Psi, A \Psi{e}_g\right) \cup \mathcal{N}_2(\bar{\eta})$, where 
	\begin{equation*}
		\mathcal{N}_2(\bar{\eta}) =\left\{x \mid \phi(x) \leq \phi^{\star}+\bar{\eta}\right\}, ~
		\bar{\eta}=\frac{\Psi L}{ c \tau \left(1-\frac{\psi}{(1+A) \Psi}\right)^2 \varepsilon \mu}\left(\varepsilon_{\rm tol}+\epsilon_f\right),
	\end{equation*}
	we have 
	$$
	\phi(x_{k+1})-\left(\phi(x^{\star})+\bar{\eta} \right)\leq \left(1-\frac{2  c \tau \left(1-\frac{\psi}{(1+A) \Psi}\right)^2}{\Psi L} \varepsilon\mu\right) \left(\phi(x_k)-\left(\phi(x^{\star})+\bar{\eta}\right)\right).
	$$
\end{theorem}
\begin{proof}
	See \cite[Theorem 5]{Irwin2023}. 
\end{proof}

Theorem \ref{thm: rate} establishes that the soft QN iterates are guaranteed to decrease the objective value at a linear rate to a neighborhood around the optimal value. However, the presence of noise in the gradient estimates eventually impedes further progress. From the definition of ${\mathcal N}_{1}$, we see that the upper bound on the residual error is proportional to $e_g^2$.  Since $\varepsilon < \psi$, it holds that $\eta \varepsilon \mu \le \frac{\psi}{L\Psi^2} \psi \mu = \frac{\mu}{L}\frac{\psi^2}{\Psi^2}$ and we can infer a convergence factor of at least $1-O(\frac{\mu}{L}\frac{\psi^2}{\Psi^2} )$. Moreover, there exists a trade-off between the convergence rate and the proximity to the optimal value that can be influenced by adjusting the free parameter $A$. Specifically, reducing $A$ results in a smaller neighborhood, but it also reduces the admissible step-size $\eta$ and consequently yields slower convergence.
In particular, in the absence of gradient computation noise, i.e., when $e_g=0$, linear convergence to the optimal value can be attained.

The bounds presented in Theorem \ref{thm:linesearch} are more complicated to interpret, due to the influence of noise on both gradient and function evaluations. When $\mathcal{N}_2$ dominates (i.e., $x_k \in \mathcal{N}_2$ but $x_k \notin \mathcal{N}_1$), the iterates produced by the algorithm generate objective function values that converge linearly to a neighborhood surrounding the optimal value that is inescapable due to the presence of $\epsilon_f$. However, in scenarios where $\mathcal{N}_1$ dominates, the algorithm lacks theoretical guarantees to further converge towards $\mathcal{N}_2$. Additionally, the magnitude of the function value decreases ensured by the line-search impacts the convergence factors and the guaranteed precision. Decreasing the parameter $c$ and increasing $\varepsilon_{\rm tol}$ relaxes the Armijo condition more but results in a slower rate of convergence and a larger neighborhood.

\begin{remark}
	Unlike most existing stochastic QN methods that consider noise from random sampling~\cite{mokhtari2014res,wang2017stochastic,moritz2016linearly,yasuda2022addhessian,chen2023modified}, this paper addresses scenarios involving more general bounded noise. Due to the unclear characteristics of general noise, we assume only boundedness in Assumption \ref{asm:bounded noise}, resulting in convergence results, as shown in Theorems \ref{thm: rate} and \ref{thm:linesearch}, which are no better than first-order methods. 
	While some studies assume unbiased gradient estimates and employ variance reduction techniques to mitigate the impact of noise on convergence in the case of random sampling~\cite{moritz2016linearly,yasuda2022addhessian,chen2023modified}, these techniques are not applicable for general bounded noise. 
	Moreover, the use of stochastic gradients to simultaneously update a Hessian inverse approximation and define the search direction of the quasi-Newton algorithm  results in a situation where $H_k$ and $g_k$ are correlated, which poses significant challenges for analysis. 
	One way to ensure independence of $g_k$ and $H_k$ is to calculate additional gradients for generating $H_k$, instead of using the same gradients $g_k$ that are used to define the search directions and impose additional assumptions such as independence of gradient noises at different evaluations. 
    If $H_k$ and $g_k$ are uncorrelated and upper and lower bounds on $H_k$ are enforced, a soft QN algorithm can be shown to fit the framework for stochastic quasi-Newton methods of \cite{wang2017stochastic}. Hence, it can be shown to be convergent in a setting with a non-convex loss function and gradient noise bounded only in variance.
    However, using different gradient evaluations for updating $H_k$ and computing $g_k$ increases the number of gradient evaluations and limits the applicable noise scenarios. 
    As we will show in the numerical experiments, our soft QN matches the performance of stochastic gradient descent with random sampling noise (see Fig.~\ref{fig:logreg}) and significantly outperforms it with additive noise (see Fig.~\ref{fig:QP}).
\end{remark}

\section{Numerical results}

In this section, we assess the effectiveness of the proposed soft QN method on logistic regression problems, randomly generated quadratic problems, and problems from the CUTEst collection~\cite{CUTEst}. The type of noise is different in all these tests: In the case of logistic regression, the noise arises due to random sampling of data. For both the quadratic problems and the CUTEst problems, artificial random noise is generated and added to the gradient. For the CUTEst problems, noise is added to the objective function as well. For the quadratic problems, the noise is sampled from the standard normal distribution, while for the CUTEst problems, it is sampled uniformly from a hypersphere. For all the tests, we used a constant $\alpha_k$ for soft QN to keep it simple. For SP-BFGS, we used methods suggested in \cite{Irwin2023} to determine $\beta_k$. Different methods for step-size selection are also used in the three different tests. For all the quasi-Newton methods used, the initial  Hessian inverse approximation is an identity matrix. Since all the problems are stochastic, the algorithm does not have access to the true objective function and its gradient, but approximations of them. When we describe the stochastic problems below, we use the same notation as in (\ref{eq:noisy-loss-g}): $\phi(x)$ is the objective function that we want to optimize, $\nabla \phi(x)$ is its gradient, while $f(x)$ and $g(x)$ are the stochastic approximations of $\phi(x)$ and $\nabla \phi(x)$ available to the solver.

\subsection{Logistic regression}

In the first setting, we consider noise that originates from random sampling. To do so, we let the algorithms solve a logistic regression problem
\begin{equation*}
	\begin{aligned}
		&\underset{x\in \mathbb{R}^{n+1}}{\mbox{minimize}} ~ \phi(x) = \frac{1}{N} \sum_{i=1}^N \phi(x;\{z_i,y_i\}) + \rho \|x_{1:n}\|^2, \\
		&\text{with}~  \phi(x;\{z_i,y_i\}):=\log(1+\exp(-y_i(x_0 + x_{1:n}^T z_i))).
	\end{aligned}
\end{equation*}
Here, $x\in \mathbb{R}^{n+1}$ is the decision variable whose first component, $x_0\in \mathbb{R}$, is the model weight for the intercept and the remaining components, $x_{1:n}\in \mathbb{R}^n$, are feature weights. Further, $(z_i, y_i)$ are the feature-label pairs in the training data, where   
$z_i \in \mathbb{R}^n$ are features and  $y_i \in \{-1,1\}$ are the associated labels. The scalar $\rho>0$ is a regularization parameter. 

To compute gradient approximations, we randomly sample a  subset $\mathcal{S} \subset \{1,...,N\}$ with cardinality $|\mathcal{S}|$ and approximate the full gradient by the stochastic gradient $g(x)$ given by $g(x):=\frac{1}{|\mathcal{S}|} \sum_{\mathcal{S}_i\in \mathcal{S}} \nabla\phi(x;(y_{\mathcal{S}_i},z_{\mathcal{S}_i}))$.  
We use the data set ijcnn1  from libsvm \cite{chang2011libsvm}, where  $n=22$ and $N=49990$. The data was pre-processed using feature-wise 2-norm normalization.

We compare our soft QN with SP-BFGS, SGD, and BFGS with gradients computed by our stochastic approximation, which we refer to as stochastic BFGS. We set $|\mathcal{S}|=1000$ and $\rho=0.1$.  
For soft QN, we set the penalty parameter $\alpha_k=0.5$. For SP-BFGS, we set the penalty parameter  $\beta_k = 0.1 \|s_k\|_2 + 10^{-10}$. As suggested in \cite{Irwin2023}, such $\beta_k$ can adaptively relax the secant equation when $s_k^Ty_k$ is dominated by noise (corresponding to small $\|s_k\|_2$
) and tighten the secant equation when $s_k^Ty_k$ is dominated by curvature information (corresponding to large $\|s_k\|_2$). Both the constant value of $\alpha_k$ and the coefficient of $\|s_k\|_2^2$ in $\beta_k$ were determined by tuning experiments conducted for one particular realization of the random sampling.  For all the algorithms, we set the step-size $\eta_k=0.1$. 
We conduct a Monte Carlo simulation with 100 trials, in which we provide the algorithms with only sampling-based gradient approximations, while also recording the actual gradient norm. 
For each trial and each method, the iterate vector is initialized to be a vector of all zeros.

\begin{figure}[h]
	\centering
	\includegraphics[width = 1 \textwidth]{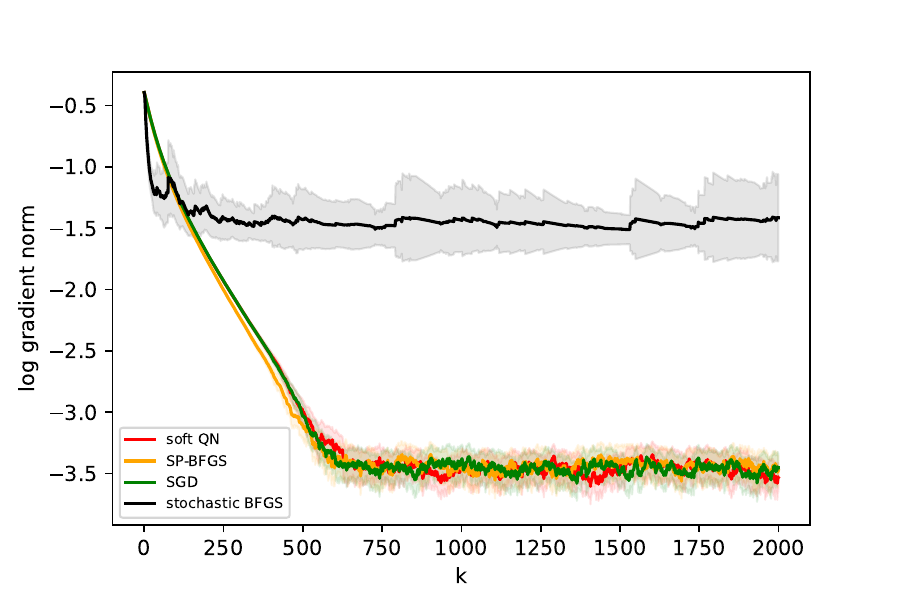}
	\caption{~Logistic regression experiments. The 10-logarithm of the gradient norm is plotted against iteration number. Solid lines indicate the mean over all instances in the Monte Carlo simulation, while shaded areas represent three standard deviation confidence intervals. }
	\label{fig:logreg}
\end{figure}
Fig.~\ref{fig:logreg} depicts the average 10-logarithm of the gradient norm for a certain iteration number, with shaded areas around the graphs representing values within three sample standard deviations of the estimator. It shows that all the other algorithms outperform stochastic BFGS. Although stochastic BFGS manages to decrease the gradient, the adaptions made to derive SP-BFGS and soft QN led to an improvement. We did not find a statistically significant difference between the performance of SP-BFGS, soft QN, and SGD in this test. This stands in contrast to the results of the subsequent experiments, which show a clear advantage for soft QN. We believe that this is due to the regularization and the pre-processing of the data, which  makes the logistic regression problem well-conditioned and limits the advantage of methods that take the curvature of the objective function into account with  Hessian inverse approximations. 

\subsection{Quadratic problems}
In the second setting, we test soft QN on problems with noise that has bounded variance, added in the gradient evaluations.  We generate quadratic problems on the form
\begin{equation}\label{eq:QP}
	\underset{x\in \mathbb{R}^n}{\mbox{minimize}} ~ \phi(x):= \frac{1}{2}x^THx + b^Tx.    
\end{equation}

%
The procedure for generating these problems is based 
on 
\cite{RandomQP}. We construct the matrix $H\in {\mathbb R}^{n\times n}$ by first defining its eigenvectors from the columns of the factor Q in the QR decomposition of a random matrix where each element follows a standard normal distribution. We then fix the minimum eigenvalue of $H$ to be 0.01 and the maximum eigenvalue to be 1, and draw the remaining $n-2$ eigenvalues uniformly from the interval $[0.01, 1]$. Finally, we choose the vector $b\in \mathbb{R}^n$ in such a way that the optimal solution to the problem \eqref{eq:QP} is ${\bf 1}$, representing a vector with all elements equal to 1. In our tests, we use $n=100$ variables. 

%
%
We consider a scenario where the algorithm has access to a gradient approximation $g(x)$ that is perturbed by noise $n_g \sim {\mathcal N}(0, I)$,
\begin{equation*}
	g(x)= \nabla \phi(x) + n_g.     
\end{equation*}
As shown in \cite{Polyak}, for unconstrained quadratic problems with additive normal noise on the gradients, Newton's method with the exact Hessian, using a step-size $\eta_k = 1/k$ where $k$ is the iteration number, is an asymptotically optimal algorithm.
Therefore, we use Newton's method as a baseline, utilizing a precise Hessian matrix but with noisy gradient $g(x)$. In addition, we compare the proposed soft QN with SGD, stochastic BFGS, and SP-BFGS.  Note that the choice of additive normal noise is made to ensure the optimality of Newton's method. While the noise components are theoretically unlimited in this case, they are limited in any practical implementation. Tests with bounded noise will be presented in the next section.

For soft QN, we set the penalty parameter $\alpha_k=10^{-4}$. For SP-BFGS, we set the penalty parameter $\beta_k=10^{-2}$ if $s_k^T y_k \geq0$ and $\beta_k=\frac{-0.9}{ s_k^T y_k}$ if $s_k^T y_k<0$. Here, we apply the SP-BFGS relaxation of the positive curvature condition~\cite[Equation~(51)]{Irwin2023} when $s_k^Ty_k<0$. The constant value of $\alpha_k$ and the value of $\beta_k$ used when $s_k^T y_k \geq 0$ were determined by tuning experiments. For stochastic BFGS, we implement the fail-safe measure that if $s_k^T y_k <0$ due to random noise, we do not update the  Hessian inverse approximation. For all algorithms, we use a diminishing step-size  $\eta_k=\frac{1}{k}$. We run a Monte Carlo simulation with 100 trials, with the initial variable vector set to ${\bf 0}$ in each trial. Note that from the way that the problems are generated, the squared norm of the initial true gradient is upper bounded by $n$, and the expected squared norm of the noise vector is $n$ throughout the iterations. Thus, noise has a significant impact on the gradient estimate even in the initial iteration. This explains why the value of $\alpha_k$ obtained in the initial tuning experiments was significantly lower than the one obtained for logistic regression: a small penalty parameter is needed in order to avoid overfitting to the noise. The results of our Monte Carlo simulations are shown in Fig.~\ref{fig:QP}. Each graph shows the average 10-logarithm of the normalized suboptimality, i.e. $\log_{10}((\phi(x_k)-\phi({\bf 1}))/(\phi({\bf 0})-\phi({\bf 1})))$, with shaded areas representing three standard deviation confidence intervals for the mean estimates.

Newton's method performs the best due to its use of accurate Hessians; stochastic BFGS actually increases the suboptimality by several orders of magnitude, as its use of noisy gradients fails to effectively capture curvature information. SGD converges slowly due to its reliance on only first-order information, and the problem is generated to have a high condition number. SP-BFGS and the proposed soft QN are faster than SGD since the second-order information accelerates convergence. As the confidence intervals do not overlap, we see that soft QN outperforms SP-BFGS with statistical significance.

\begin{figure}[H]
	\centering
	\includegraphics[width = 1\textwidth]{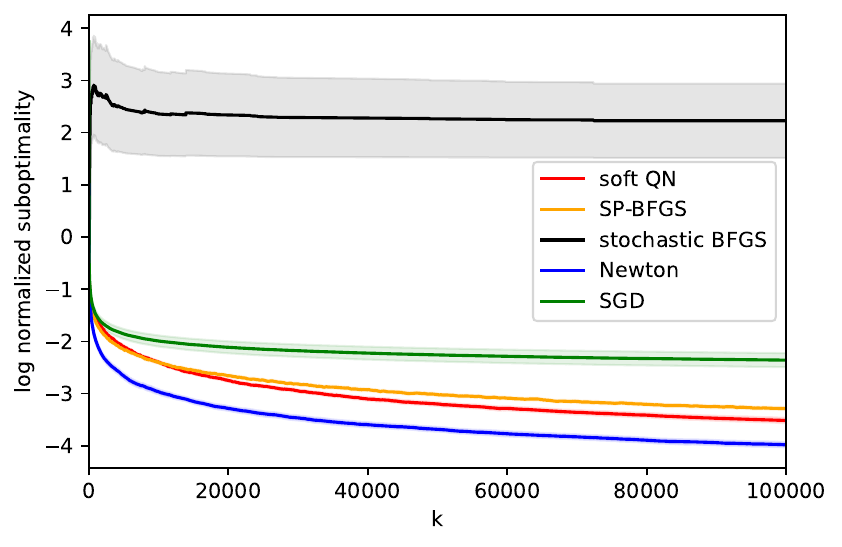}
	\caption{ ~Experiments on quadratic problems. The 10-logarithm of the normalized suboptimality is plotted against iteration number. Solid lines show the mean over all instances in the Monte Carlo simulation, shaded areas represent three standard deviation confidence intervals. }
	\label{fig:QP}
\end{figure}
\subsection{CUTEst problems}
In the third setting, we consider bounded noise in the function and gradient evaluations. Specifically, we let the solver have access to $f(x)$ and $g(x)$, where 
\begin{equation}
	f(x) := \phi(x) + n_f(x), ~g(x) := \nabla \phi(x) + n_g(x).   
\end{equation}
Here, $n_f$ is sampled randomly from the uniform distribution over the interval $[-e_f,e_f]$, and $n_g$ is sampled uniformly from the $n$-dimensional sphere with radius $e_g$. We compare the performance of soft QN and SP-BFGS for solving a set of problems from the CUTEst suite~\cite{CUTEst} The names and the dimensions of the problems are listed in Table~\ref{tab:problem_list} in Appendix~\ref{app:cutest}.  The updates of soft QN and SP-BFGS were combined with a backtracking line-search procedure similar to the one described in~\cite{Irwin2023}, summarized in Algorithm~\ref{alg:line_search}.  

\begin{algorithm}
	\caption{Line-search procedure}
	\label{alg:line_search}
	\begin{algorithmic}[1]
		\Procedure{line-search under noise~}{$p_k, x_k,g(x_k),  \varepsilon_{tol},c,\tau,T,f(\cdot)$}
		\State $\eta_k \gets 1$
		\State $t \gets 0$
		\While{$f(x_k + \eta_k p_k) > f(x_k) + \eta_k c p_k^T g(x_k) + 2 \varepsilon_{tol} \And t < T$ }
		\State $\eta_k \gets \tau \eta_k$
		\State $t \gets t + 1$
		\EndWhile
		\If{$f(x_k + \eta_k p_k) < f(x_k) + 2 \varepsilon_{tol}$}
		\State \Return $\eta_k$
		\Else
		\State \Return 0
		\EndIf
		\EndProcedure
	\end{algorithmic}
\end{algorithm}

To scale the noise according to the problem, we set $e_f = 10^{-4} |\phi(x_0)|$ and $e_g = 10^{-4} \|\nabla \phi(x_0) \|_2$, where $x_0$ is the initial point. We set $\tau = 0.5$, $c=10^{-4}$, $T=45$, and $\varepsilon_{tol} = e_f$. 
As for the values of $\alpha_k$ for soft QN and $\beta_k$ for SP-BFGS, we utilize different strategies. After some test runs on the problem ARWHEAD, we determined that $\alpha_k = 10^6$ for all the test problems. In contrast to the experiments with quadratic problems, the noise is low compared to the norm of the initial gradient, allowing a significantly larger $\alpha_k$ without underfitting curvature information. For $\beta_k$, we set $\beta_k = \frac{10^8}{\epsilon_g} \|s_k\|_2 + 10^{-10}$, which has been heuristically discovered by~\cite{Irwin2023} to work well in practice for various CUTEst problems. However, if this choice would result in a non-positive definite matrix, i.e., $s_k^T y_k \leq -\frac{1}{\beta_k}$, we instead refrain from updating the  Hessian inverse approximation for that iteration.  

For each test problem, we run the algorithms 30 times and track the measure
\begin{equation*}
	\Delta^i_j := \phi(x^i_j) - \phi^{\star},
\end{equation*}
where $x^i_j$ is the variable vector after $j$ function evaluations for test-run $i$, and $  \phi^{\star}$ is the optimal value that is available in the file defining the CUTEst problem. The initial point from which the iterations are started is also found in the same file.
Note that while we use the exact function values for the purpose of tracking suboptimality to evaluate performance, the algorithms in the test are only given access to noisy function values.
%
The line-search that we apply involves a tolerance term for the noise and does not guarantee descent in each step, so the suboptimality may actually increase from one instance to another. 
Nevertheless, we observe a clear downward trend in suboptimality in the test runs; see Figs~\ref{fig:DIXMAANA_soft QN} and~\ref{fig:DIXMAANA_SP-BFGS}, which show the performance for soft QN and SP-BFGS on the test problem DIXMAANA. 

In Figures~\ref{fig:DIXMAANA_soft QN} and~\ref{fig:DIXMAANA_SP-BFGS}, we show the median of all $\Delta^i_j, \ i \in \{1,\dots,30\}$ against the number of function evaluations $j$, and we show the range of all suboptimalities as well as the middle quartiles. The reason for using this way of visualizing the results, in contrast to previous plots, is the varying number of function evaluations in each iteration of the algorithm. For a specific number $j$ of function evaluations, the algorithm in test run $i$ might still be in the process of performing a line-search. The value of $x^i_j$ is then the last iterate before $j$ function evaluations has been performed, at which the algorithm had finished a line-search. Since the number of function evaluations performed before computing $x^i_j$ can vary considerably between test runs, one can expect the distribution of $\Delta^i_j$ to be skewed, making the median a better representation of the typical suboptimality than the mean in this case. Showing the quartiles as well as the full range of the distribution of $\Delta^i_j$ gives an idea of the skewness of the distribution.

For the test perfomed on the DIXMAANA problem, we can see that soft QN has a consistently better performance than SP-BFGS; Its median suboptimality decreases faster, the variability of its suboptimality is lower, and at the final iteration, the highest suboptimality for all of the test runs for soft QN is lower than the lowest suboptimality for SP-BFGS. 
\begin{figure}
	\centering
	\begin{subfigure}[b]{0.49\textwidth}
		\centering
		\includegraphics[width=7.4cm]{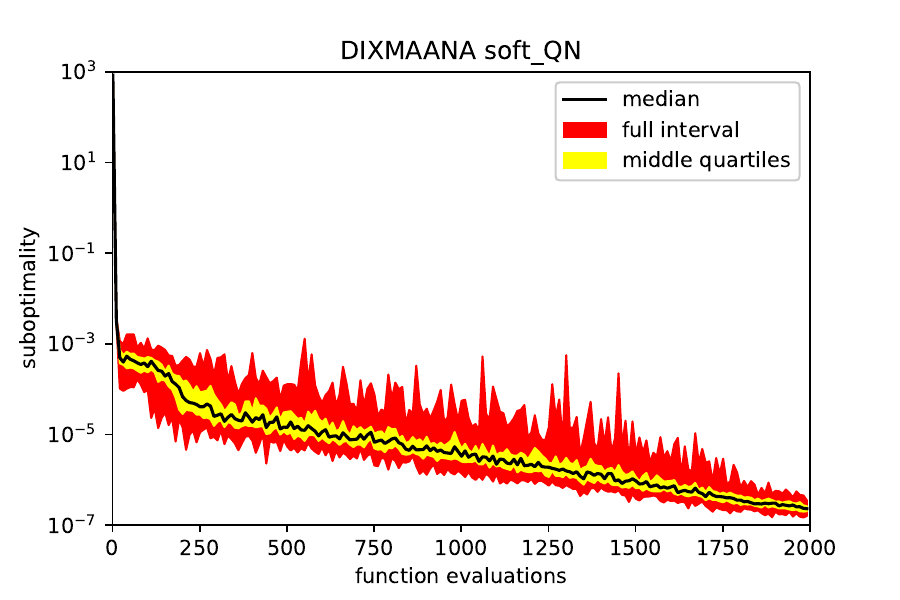}
		\caption{soft QN}
		\label{fig:DIXMAANA_soft QN}
	\end{subfigure}
	\hfill
	\begin{subfigure}[b]{0.49\textwidth}
		\centering
		\includegraphics[width=7.4cm]{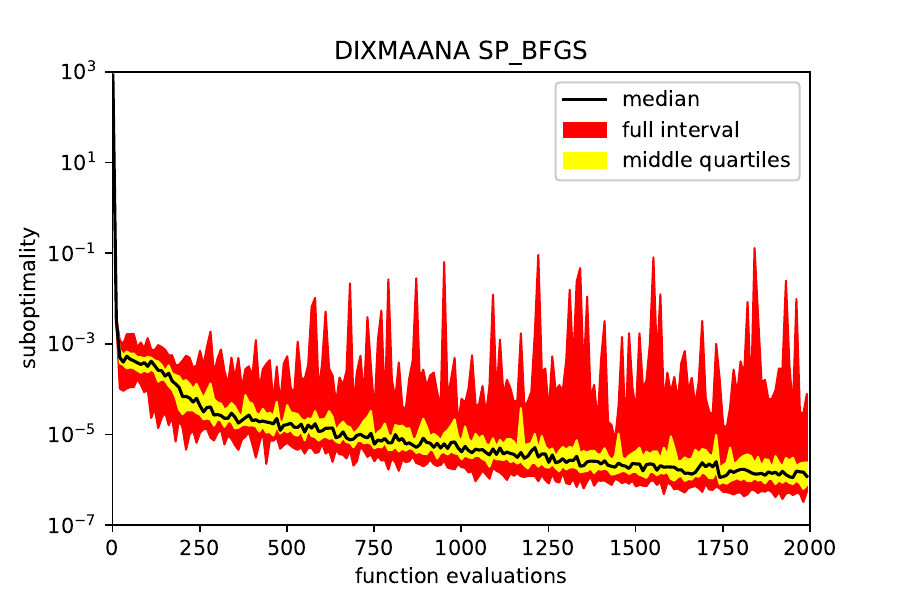}
		\caption{SP-BFGS}
		\label{fig:DIXMAANA_SP-BFGS}
	\end{subfigure}
	\caption{~Comparision of soft QN and SP-BFGS on the CUTEst-problem DIXMAANA. Interval of suboptimality for different test runs, plotted on a logarithmic scale against number of function evaluations.} 
	\label{fig-78}
\end{figure}

For all the tests with the CUTEst test problems, we terminate the algorithms as soon as 2000 function evaluations have been exceeded, interrupting any line-searches and determining whether the previous iterate or the current trial iterate in the line-search should be considered the final one, in the same way as if the usual maximum number of function evaluations in the line search $T$ has been exceeded. We then compute summary statistics for of $\Delta^i_{2000}$ for all the test runs. The results are given in Tables~\ref{tab: CUTEst soft QN} and~\ref{tab: CUTEst SP-BFGS} in Appendix \ref{app:cutest}.

\section{Conclusion}
We have proposed a novel quasi-Newton method to address the challenges of minimizing loss functions in the presence of general bounded noise. By replacing the traditionally hard secant condition with a penalty on its violation, we developed a recursive Hessian inverse approximation that is guaranteed to be positive definite for all penalty parameters, even for non-convex loss functions.  This soft QN update exhibits several attractive properties: it recovers BFGS under specific limits, treats positive and negative curvature equally, and is scale invariant. These features collectively enhance the efficacy of soft QN in noisy environments.  For strongly convex loss functions, we have detailed an algorithm that has linear convergence to a neighborhood of the optimal solution. Numerical experiments demonstrate the superiority of soft QN compared with existing algorithms, validating its practical effectiveness and highlighting its potential for various optimization tasks in noisy scenarios. 

Although this paper has provided a relatively complete treatment of the soft QN update, several problems warrant further attention. These include exploring ways to adapt the penalty parameter dynamically during iterations to help escape saddle points, as well as developing statistically efficient techniques in the presence of random function and gradient noise. Furthermore, designing a limited memory version of soft QN would be crucial for extending its applicability to large-scale settings. We hope to return to these issues in our future work.


\section*{Disclosure statement}

No potential conflict of interest was reported by the authors.

\section*{Funding}

This work was supported in part by Swedish Research Council (VR) under grant 2019-05319, by the Knut and Alice Wallenberg Foundation, and Digital Futures.

\section*{Notes on contributor(s)}

\noindent\emph{\textbf{Erik Berglund}} received the B.Sc. degree in Engineering Physics and the M.Sc. degree in Applied and Computational Mathematics from KTH Royal Institute of Technology, Stockholm, Sweden, in 2015 and 2017 respectively. He is currently a PhD student in the Division of Decision and Control Systems at KTH Royal Institute of Technology. His research is focused on the use of Hessians and Hessian approximations in optimization algorithms.  \\

\noindent\emph{\textbf{Jiaojiao Zhang}} 
received the B.E. degree in automation from School of Automation, Harbin Engineering University, Harbin, China, in 2015.  She received the master degree in control theory and control engineering from University of Science and Technology of China, Hefei, China, in 2018.  She received the Hong Kong PhD Fellowship Scheme (HKPFS) in August 2018. She received Ph.D. degree in the Department of Systems   Engineering and Engineering Management, Chinese University of Hong Kong, in 2022.  She is currently a postdoctoral researcher in the Division of Decision and Control Systems at KTH Royal Institute of Technology, Stockholm, Sweden. Her current research interests include distributed optimization and algorithm design.\\

\noindent\emph{\textbf{Mikael Johansson}} received the M.Sc. and Ph.D. degrees in electrical engineering from Lund University, Lund, Sweden, in 1994 and 1999, respectively. He held postdoctoral positions with Stanford University, Stanford, CA, USA, and the University of California at Berkeley, Berkeley, CA, USA, before joining the KTH Royal Institute of Technology, Stockholm, Sweden, in 2002, where he is currently a Full Professor. He has played a leading role in several national and international research projects on optimization, control, and communications. He has coauthored two books and over 200 papers, several of which are highly cited and have received recognition in terms of best paper awards. His research interests revolve around large-scale and distributed optimization, autonomous decision-making, control, and machine learning. Dr. Johansson has served on the Editorial Boards of Automatica and the IEEE Transactions on Control of Network Systems and on the program committee for several top conferences organized by IEEE and ACM.

\bibliographystyle{acm}
\bibliography{OMS-soft_QN_preprint}

\appendix
\section{}\label{app:cutest}

The problem names and dimensions of the considered test problems in the CUTEst suite are given in Table \ref{tab:problem_list}. 
Results for all the CUTEst test problems are summarized in Tables~\ref{tab: CUTEst soft QN} and \ref{tab: CUTEst SP-BFGS}. In these tables, we compare soft QN and SP-BFGS by various metrics: The minimum, maximum, mean, median, and sample variance of all $\Delta^i_{2000}$ for each problem. In each table, we have marked each cell with a lower value than the corresponding cell in the table for the other algorithm in green. From these tables, one can see that out of the 32 test problems, soft QN had a lower min $\Delta^i_{2000}$ on 20, a lower max $\Delta^i_{2000}$ on 20, a lower mean $\Delta^i_{2000}$ on 19, a lower median $\Delta^i_{2000}$ on 22 and a lower variance of $\Delta^i_{2000}$ on 20 of the problems. This indicates that soft QN has a competitive performance compared to SP-BFGS in this setting.

\begin{table}[ht]
	\centering
	\begin{tabular}{l|l|l|l|l|l}
		\hline
		Problem name & $n$ & Problem name & $n$ & Problem name & $n$ \\ \hline
		ARWHEAD & 100 & DIXMAANI & 90 & GENROSE & 100 \\ \hline
		BDQRTIC & 100 & DIXMAANJ & 90 & MOREBV & 100 \\ \hline
		CRAGGLVY & 100 & DIXMAANK & 90 & NONDIA & 100 \\ \hline
		DIXMAANA & 90 & DIXMAANL & 90 & NONDQUAR & 100 \\ \hline
		DIXMAANB & 90 & DIXMAANM & 90 & QUARTC & 100 \\ \hline
		DIXMAANC & 90 & DIXMAANN & 90 & SPARSQUR & 100 \\ \hline
		DIXMAAND & 90 & DIXMAANO & 90 & TQUARTIC & 100 \\ \hline
		DIXMAANE & 90 & DIXMAANP & 90 & TRIDIA & 100 \\ \hline
		DIXMAANF & 90 & EIGENALS & 110 & WATSON & 31 \\ \hline
		DIXMAANG & 90 & EIGENBLS & 110 & WOODS & 100 \\ \hline
		DIXMAANH & 90 & EIGENCLS & 30 & & \\ \hline
	\end{tabular}
	\caption{The name and dimension of test problems from the CUTEst suite.}
	\label{tab:problem_list}
\end{table}

\begin{table}[H]
	\begin{tabular}{l|lllll}
		Problem name  & Min $\Delta_{2000}^i $ & Max $\Delta_{2000}^i$ & Mean $\Delta_{2000}^i $ & Median $\Delta_{2000}^i$ & Var $\Delta_{2000}^i$\\ \hline
		ARWHEAD  & \cellcolor[HTML]{92D050}9.22E-07 & 1.52E-04                         & \cellcolor[HTML]{92D050}7.45E-06 & \cellcolor[HTML]{92D050} 2.17E-06                         & 7.19E-10                         \\
		BDQRTIC  & 5.76E-04                         & 2.07                             & 1.07E-01                         & \cellcolor[HTML]{92D050}1.36E-03 & 1.15E-01                         \\
		CRAGGLVY & \cellcolor[HTML]{92D050}2.67E-03 & 8.40E+01                         & 3.12                             & \cellcolor[HTML]{92D050}1.16E-02 & 2.27E+02                         \\
		DIXMAANA & \cellcolor[HTML]{92D050}1.47E-07 & \cellcolor[HTML]{92D050}2.84E-07 & \cellcolor[HTML]{92D050}2.23E-07 & \cellcolor[HTML]{92D050}2.28E-07 & \cellcolor[HTML]{92D050}1.62E-15 \\
		DIXMAANB & \cellcolor[HTML]{92D050}5.81E-07 & \cellcolor[HTML]{92D050}9.25E-06 & \cellcolor[HTML]{92D050}1.57E-06 & \cellcolor[HTML]{92D050}1.11E-06 & \cellcolor[HTML]{92D050}2.44E-12 \\
		DIXMAANC & 3.42E-06                         & \cellcolor[HTML]{92D050}2.05E-04 & \cellcolor[HTML]{92D050}1.83E-05 & \cellcolor[HTML]{92D050}6.35E-06 & \cellcolor[HTML]{92D050}1.40E-09 \\
		DIXMAAND & \cellcolor[HTML]{92D050}1.45E-05 & \cellcolor[HTML]{92D050}1.15E-02 & \cellcolor[HTML]{92D050}4.98E-04 & \cellcolor[HTML]{92D050}3.11E-05 & \cellcolor[HTML]{92D050}4.25E-06 \\
		DIXMAANE & \cellcolor[HTML]{92D050}3.37E-07 & \cellcolor[HTML]{92D050}8.90E-07 & \cellcolor[HTML]{92D050}5.48E-07 & \cellcolor[HTML]{92D050}5.36E-07 & \cellcolor[HTML]{92D050}2.05E-14 \\
		DIXMAANF & \cellcolor[HTML]{92D050}1.46E-06 & \cellcolor[HTML]{92D050}1.20E-05 & \cellcolor[HTML]{92D050}3.34E-06 & \cellcolor[HTML]{92D050}2.39E-06 & \cellcolor[HTML]{92D050}5.59E-12 \\
		DIXMAANG & \cellcolor[HTML]{92D050}2.25E-04 & \cellcolor[HTML]{92D050}1.34E-03 & \cellcolor[HTML]{92D050}8.15E-04 & \cellcolor[HTML]{92D050}8.30E-04 & \cellcolor[HTML]{92D050}8.37E-08 \\
		DIXMAANH & 3.91E-05                         & \cellcolor[HTML]{92D050}1.41E-03 & \cellcolor[HTML]{92D050}1.69E-04 & 7.50E-05                         & \cellcolor[HTML]{92D050}6.95E-08 \\
		DIXMAANI & \cellcolor[HTML]{92D050}1.37E-04 & \cellcolor[HTML]{92D050}2.54E-04 & \cellcolor[HTML]{92D050}1.95E-04 & \cellcolor[HTML]{92D050}1.97E-04 & \cellcolor[HTML]{92D050}9.75E-10 \\
		DIXMAANJ & 2.36E-03                         & 9.03E-03                         & 5.53E-03                         & 5.46E-03                         & 3.96E-06                         \\
		DIXMAANK & \cellcolor[HTML]{92D050}5.61E-04 & \cellcolor[HTML]{92D050}2.96E-03 & 1.73E-03                         & \cellcolor[HTML]{92D050}1.76E-03 & \cellcolor[HTML]{92D050}4.47E-07 \\
		DIXMAANL & \cellcolor[HTML]{92D050}2.03E-03 & 2.21E-02                         & 6.45E-03                         & 5.67E-03                         & 1.26E-05                         \\
		DIXMAANM & 1.42E-04                         & \cellcolor[HTML]{92D050}2.20E-04 & \cellcolor[HTML]{92D050}1.85E-04 & \cellcolor[HTML]{92D050}1.89E-04 & \cellcolor[HTML]{92D050}4.76E-10 \\
		DIXMAANN & \cellcolor[HTML]{92D050}1.84E-04 & \cellcolor[HTML]{92D050}2.65E-04 & \cellcolor[HTML]{92D050}2.20E-04 & \cellcolor[HTML]{92D050}2.19E-04 & \cellcolor[HTML]{92D050}3.07E-10 \\
		DIXMAANO & 1.64E-04                         & \cellcolor[HTML]{92D050}2.16E-04 & \cellcolor[HTML]{92D050}1.87E-04 & \cellcolor[HTML]{92D050}1.87E-04 & \cellcolor[HTML]{92D050}1.71E-10 \\
		DIXMAANP & 2.44E-04                         & 2.40E-03                         & 5.38E-04                         & \cellcolor[HTML]{92D050}4.81E-04 & 1.32E-07                         \\
		EIGENALS & 3.69E-06                         & \cellcolor[HTML]{92D050}5.88E-05 & \cellcolor[HTML]{92D050}2.25E-05 & 1.93E-05                         & \cellcolor[HTML]{92D050}2.77E-10 \\
		EIGENBLS & \cellcolor[HTML]{92D050}2.81E-08 & 8.24E-07                         & 1.58E-07                         & \cellcolor[HTML]{92D050}9.19E-08 & 3.02E-14                         \\
		EIGENCLS & \cellcolor[HTML]{92D050}1.47E-09 & \cellcolor[HTML]{92D050}5.62E-09 & \cellcolor[HTML]{92D050}3.83E-09 & \cellcolor[HTML]{92D050}3.93E-09 & \cellcolor[HTML]{92D050}9.65E-19 \\
		GENROSE  & \cellcolor[HTML]{92D050}6.80E-09 & \cellcolor[HTML]{92D050}2.71E-08 & \cellcolor[HTML]{92D050}1.61E-08 & \cellcolor[HTML]{92D050}1.52E-08 & \cellcolor[HTML]{92D050}2.58E-17 \\
		MOREBV   & 8.55E-07                         & 8.56E-07                         & 8.55E-07                         & 8.55E-07                         & \cellcolor[HTML]{92D050} 1.08E-20                         \\
		NONDIA   & \cellcolor[HTML]{92D050}1.78E-04 & 2.57E-02                         & 1.28E-03                         & \cellcolor[HTML]{92D050}3.85E-04 & 2.06E-05                         \\
		NONDQUAR & 4.42E-04                         & \cellcolor[HTML]{92D050}6.83E-04 & \cellcolor[HTML]{92D050}5.76E-04 & 5.78E-04                         & \cellcolor[HTML]{92D050}4.02E-09 \\
		QUARTC   & \cellcolor[HTML]{92D050}1.27E+02 & 2.87E+05                         & 1.64E+04                         & \cellcolor[HTML]{92D050}2.16E+02 & 3.12E+09                         \\
		SPARSQUR & \cellcolor[HTML]{92D050}5.00E-05 & 3.05E-02                         & 1.59E-03                         & 1.10E-03                         & 3.18E-05                         \\
		TQUARTIC & 1.74E-10                         & \cellcolor[HTML]{92D050}1.98E-09 & 1.06E-09                         & 1.06E-09                         & 2.11E-19                         \\
		TRIDIA   & \cellcolor[HTML]{92D050}2.53E-07 & \cellcolor[HTML]{92D050}1.49E-06 & \cellcolor[HTML]{92D050}4.66E-07 & \cellcolor[HTML]{92D050}4.60E-07 & \cellcolor[HTML]{92D050}5.46E-14 \\
		WATSON   & \cellcolor[HTML]{92D050}2.39E-05 & \cellcolor[HTML]{92D050}3.17E-04 & \cellcolor[HTML]{92D050}1.29E-04 & 1.16E-04                         & \cellcolor[HTML]{92D050}5.51E-09 \\
		WOODS    & 4.23E-02                         & 6.30E-01                         & 1.17E-01                         & 1.00E-01                         & 1.02E-02                       
	\end{tabular}
	\caption{ Results from the test runs with soft QN on a set of CUTEst problems. Instances where soft QN performed better than SP-BFGS are marked in green.}
 \label{tab: CUTEst soft QN}
\end{table}

\begin{table}[H]
	\begin{tabular}{l|lllll}
		Problem name  & Min $\Delta_{2000}^i $ & Max $\Delta_{2000}^i$ & Mean $\Delta_{2000}^i $ & Median $\Delta_{2000}^i$ & Var$\Delta_{2000}^i$\\ \hline
		ARWHEAD  & 2.05E-06                         & \cellcolor[HTML]{92D050}5.59E-05 & 9.43E-06                         & 5.11E-06 & \cellcolor[HTML]{92D050}1.24E-10  \\
		BDQRTIC  & \cellcolor[HTML]{92D050}5.53E-04 & \cellcolor[HTML]{92D050}5.86E-02 & \cellcolor[HTML]{92D050}4.61E-03 & 1.48E-03                          & \cellcolor[HTML]{92D050}1.11E-04  \\
		CRAGGLVY & 3.35E-03                         & 8.40E+01                         & 3.12                             & 1.29E-02                          & 2.27E+02                          \\
		DIXMAANA & 3.86E-07                         & 3.81E-04                         & 1.47E-05                         & 1.20E-06                          & 4.65E-09                          \\
		DIXMAANB & 9.57E-07                         & 3.12E-05                         & 4.54E-06                         & 2.50E-06                          & 3.11E-11                          \\
		DIXMAANC & \cellcolor[HTML]{92D050}2.97E-06 & 6.92E-04                         & 4.27E-05                         & 9.50E-06                          & 1.58E-08                          \\
		DIXMAAND & 1.55E-05                         & 3.00E-02                         & 1.13E-03                         & 3.42E-05                          & 2.89E-05                          \\
		DIXMAANE & 7.61E-07                         & 5.26E-05                         & 4.61E-06                         & 1.96E-06                          & 9.02E-11                          \\
		DIXMAANF & 1.92E-06                         & 9.11E-05                         & 1.06E-05                         & 4.84E-06                          & 3.06E-10                          \\
		DIXMAANG & 2.90E-04                         & 1.11E-01                         & 6.03E-03                         & 1.09E-03                          & 4.11E-04                          \\
		DIXMAANH & \cellcolor[HTML]{92D050}3.60E-05 & 1.19E-01                         & 4.39E-03                         & \cellcolor[HTML]{92D050}7.41E-05  & 4.55E-04                          \\
		DIXMAANI & 1.38E-04                         & 2.79E-04                         & 2.05E-04                         & 2.08E-04                          & 1.38E-09                          \\
		DIXMAANJ & \cellcolor[HTML]{92D050}2.35E-03 & \cellcolor[HTML]{92D050}8.98E-03 & \cellcolor[HTML]{92D050}5.27E-03 & \cellcolor[HTML]{92D050}4.88E-03  & \cellcolor[HTML]{92D050}3.78E-06  \\
		DIXMAANK & 6.27E-04                         & 3.14E-03                         & 1.73E-03                         & 1.80E-03                          & 5.10E-07                          \\
		DIXMAANL & 2.14E-03                         & \cellcolor[HTML]{92D050}1.06E-02 & \cellcolor[HTML]{92D050}5.72E-03 & \cellcolor[HTML]{92D050}5.57E-03  & \cellcolor[HTML]{92D050}2.67E-06  \\
		DIXMAANM & \cellcolor[HTML]{92D050}1.23E-04 & 2.36E-04                         & 1.90E-04                         & 1.93E-04                          & 7.49E-10                          \\
		DIXMAANN & 1.88E-04                         & 6.02E-03                         & 5.11E-04                         & 2.23E-04                          & 1.28E-06                          \\
		DIXMAANO & \cellcolor[HTML]{92D050}1.58E-04 & 2.67E-04                         & 1.97E-04                         & 1.94E-04                          & 3.85E-10                          \\
		DIXMAANP & \cellcolor[HTML]{92D050}2.29E-04 & \cellcolor[HTML]{92D050}6.82E-04 & \cellcolor[HTML]{92D050}4.74E-04 & 4.92E-04                          & \cellcolor[HTML]{92D050}9.86E-09  \\
		EIGENALS & \cellcolor[HTML]{92D050}2.33E-06 & 7.63E-05                         & 2.46E-05                         & \cellcolor[HTML]{92D050}1.86E-05  & 3.46E-10                          \\
		EIGENBLS & 4.65E-08                         & \cellcolor[HTML]{92D050}2.77E-07 & \cellcolor[HTML]{92D050}1.31E-07 & 1.05E-07                          & \cellcolor[HTML]{92D050}3.18E-15  \\
		EIGENCLS & 4.22E-09                         & 8.63E-08                         & 1.83E-08                         & 1.48E-08                          & 2.25E-16                          \\
		GENROSE  & 8.39E-09                         & 1.15E-07                         & 5.42E-08                         & 5.30E-08                          & 5.19E-16                          \\
		MOREBV   & \cellcolor[HTML]{92D050}4.78E-12 & \cellcolor[HTML]{92D050}3.86E-09 & \cellcolor[HTML]{92D050}2.88E-10 & \cellcolor[HTML]{92D050}9.31E-11  & 4.70E-19 \\
		NONDIA   & 2.05E-04                         & \cellcolor[HTML]{92D050}6.26E-03 & \cellcolor[HTML]{92D050}6.99E-04 & 4.38E-04                          & \cellcolor[HTML]{92D050}1.12E-06  \\
		NONDQUAR & \cellcolor[HTML]{92D050}4.40E-04 & 2.37E-03                         & 6.40E-04                         & \cellcolor[HTML]{92D050}5.68E-04  & 1.15E-07                          \\
		QUARTC   & 1.31E+02                         & \cellcolor[HTML]{92D050}4.96E+04 & \cellcolor[HTML]{92D050}2.39E+03 & 2.43E+02                          & \cellcolor[HTML]{92D050}8.24E+07  \\
		SPARSQUR & 5.76E-05                         & \cellcolor[HTML]{92D050}2.05E-03 & \cellcolor[HTML]{92D050}2.02E-04 & \cellcolor[HTML]{92D050}9.80E-05  & \cellcolor[HTML]{92D050}1.28E-07  \\
		TQUARTIC & \cellcolor[HTML]{92D050}1.06E-10 & 1.99E-09                         & \cellcolor[HTML]{92D050}7.38E-10 & \cellcolor[HTML]{92D050}7.00E-10  & \cellcolor[HTML]{92D050}1.42E-19  \\
		TRIDIA   & 1.23E-06                         & 6.28E-06                         & 3.05E-06                         & 2.76E-06                          & 1.55E-12                          \\
		WATSON   & 2.42E-05                         & 3.46E-04                         & 1.32E-04                         & \cellcolor[HTML]{92D050}1.12E-04  & 7.26E-09                          \\
		WOODS    & \cellcolor[HTML]{92D050}4.08E-02 & \cellcolor[HTML]{92D050}1.58E-01 & \cellcolor[HTML]{92D050}9.19E-02 & \cellcolor[HTML]{92D050}8.98E-02  & \cellcolor[HTML]{92D050}8.00E-04     
	\end{tabular}
	\caption{ Results from the test runs with SP-BFGS on a set of CUTEst problems. Instances where SP-BFGS performed better than soft QN are marked in green.}
 \label{tab: CUTEst SP-BFGS}
\end{table}

\end{document}